\documentclass[reqno]{amsart}
\usepackage{amsmath, amsthm, amssymb, amstext}

\usepackage{hyperref,xcolor}
\hypersetup{pdfborder={0 0 0},colorlinks}
\usepackage{enumitem}
\allowdisplaybreaks
\usepackage{graphicx}
\usepackage{caption}
\usepackage{mathtools}
\usepackage{tikz}
\usepackage{esint}

\newtheorem{theorem}{Theorem}
\newtheorem{remark}[theorem]{Remark}
\newtheorem{lemma}[theorem]{Lemma}
\newtheorem{proposition}[theorem]{Proposition}
\newtheorem{corollary}[theorem]{Corollary}
\newtheorem{definition}[theorem]{Definition}

\DeclareMathOperator*{\ess}{\text{ess}}
\DeclareMathOperator*{\diam}{\text{diam}}
\DeclareMathOperator*{\dist}{\text{dist}}

\newcommand{\N}{\mathbb{N}}
\newcommand{\R}{\mathbb{R}}
\newcommand{\h}{\mathbb{H}}
\newcommand{\eps}{\varepsilon}
\newcommand{\ph}{\varphi}
\newcommand{\into}{\int_{\Omega}}
\newcommand{\intr}{\iint_{\R^N\times\R^N}} 
\newcommand{\inta}{\iint\limits_{\substack{x,y\in\R^N \\ \abs{x-y}\leq1}}}
\newcommand{\intb}{\iint\limits_{\substack{x,y\in\R^N \\ \abs{x-y} \geq 1}}}
\newcommand{\intoo}{\iint\limits_{\Omega \times \Omega}} 
\renewcommand{\l}{\left}
\renewcommand{\r}{\right}

\def\abs#1{\left|{#1}\right|}

\numberwithin{theorem}{section}
\numberwithin{equation}{section}

\title[Brezis-Nirenberg and logistic  problem for Logarithmic Laplacian]{The Brezis-Nirenberg and logistic  problem for the Logarithmic Laplacian}

\author[R. Arora]{Rakesh Arora}
\address[R. Arora]{ Department of Mathematical Sciences, Indian Institute of Technology Varanasi (IIT-BHU), Uttar Pradesh 221005, India}
\email{rakesh.mat@iitbhu.ac.in, arora.npde@gmail.com}

\author[J. Giacomoni]{Jacques Giacomoni}
\address[J. Giacomoni]{
LMAP, UMR E2S-UPPA CNRS 5142, Ba\^timent IPRA, Avenue de l’Universit\'e F-64013 Pau, France}
\email{jacques.giacomoni@univ-pau.fr}

\author[A. Vaishnavi]{Arshi Vaishnavi}
\address[A. Vaishnavi]{Department of Mathematical Sciences, Indian Institute of Technology Varanasi (IIT-BHU), Uttar Pradesh 221005, India}
\email{arshiv1998@gmail.com}

\subjclass{35B40, 35S15, 35J60, 35R11}
\keywords{Logarithmic Laplacian, D\'iaz-Saa type inequality, Fractional Laplacian, Brezis-Nirenberg problem, Logistic equation, Orlicz Spaces}

\begin{document}
\begin{abstract}
    In this work, we study the non-local analogue of Brezis-Nirenberg and logistic type elliptic equations involving the logarithmic Laplacian and critical logarithmic non-linearity with superlinear-subcritical perturbation. 
    
     In the first part of this work, we derive new sharp, continuous and compact embeddings of nonlocal Sobolev spaces (of order zero) into Orlicz type spaces. As an application of these embeddings and variational analysis as carried out in \cite{Angeles-Saldana-2023, Santamaria-Saldana-2022}, we prove the existence of a least energy weak solution of the Brezis-Nirenberg and logistic type problem involving the logarithmic Laplacian. For the uniqueness of solution, we prove a new D\'iaz-Saa type inequality, which is of independent interest and can be applied to a larger class of problems. 
    
    In the second part of the work, depending upon the growth of non-linearity and regularity of the weight function, we study the small-order asymptotic of non-local weighted elliptic equations involving the fractional Laplacian of order $2s.$ We show that least energy solutions of a weighted non-local fractional problem with superlinear or sublinear type non-linearity converge to a non-trivial, non-negative least energy solution of a Brezis-Nirenberg type or logistic-type problem, respectively, involving the logarithmic Laplacian.
    
\end{abstract}
\maketitle
\tableofcontents
\section{Introduction}
Recently, there has been a growing interest in studying boundary value problems involving the logarithmic Laplacian. This interest has been driven by important applications together with significant advances in understanding non-local phenomena in the frame of partial differential equations (PDEs). For instance, in population dynamics \cite{Pellacci-Verzini-2018}, a logistic-type equation involving the fractional Laplacian of order $2s$ is employed to describe the movement of species. The parameter $s$ accounts for different dispersal strategies, with smaller values of $s$ describing nearly static populations capable of moving long distances swiftly and larger values of $s$ corresponding to highly dynamic species that mostly move short distances. Moreover, in \cite{Pellacci-Verzini-2018}, it is shown that a very small order $s$ models the best strategy for survival if the habitat is not too fragmented or not too hostile on an average. Similarly, a small value of the exponent $s$ yields an optimal choice in various applications, including optimal control \cite{Sprekels-Valdinoci-2017}, population dispersal strategies \cite{Pellacci-Verzini-2018}, and fractional image denoising \cite{Antil-Bartels-2017}. This line of investigation (leading to the study of emerging  asymptotics as $s\to 0^+$ in different models) is new and of independent interest from a theoretical point of view (see \cite{Angeles-Saldana-2023, Chen-Weth-2019, Feulefack-Jarohs-2022, Jarohs-Saldana-Weth-2020, Santamaria-Saldana-2022}).

In this article, we aim to study the following non-local elliptic equation involving the logarithmic Laplacian
\begin{equation}\label{prob:logarithmic-laplace}
\tag{$\mathcal{P}_{\Delta,\sigma}$}
    L_\Delta u  = f(x,u) + \sigma u \ln\abs{u}  \quad 
    \text{in } \Omega, \quad u= 0 \quad \text{in } \R^N\setminus\Omega,
\end{equation}
where $\Omega\subseteq \R^N$, $N \geq 1$ is a bounded domain with Lipschitz boundary, $f: \Omega \times \mathbb{R} \to \mathbb{R}$ is a given function and $\sigma \in \mathbb{R}$. Here, the operator $L_\Delta$ denotes the logarithmic Laplacian, {\it i.e.,} the pseudo-differential operator with Fourier symbol $2 \ln |\xi|,$ which can also be seen as a first-order expansion of the fractional Laplacian (the pseudo-differential operator with Fourier symbol $|\xi|^{2s}$). In particular, for $u \in C_c^2(\mathbb{R}^N)$ and $x \in \mathbb{R}^N,$
\begin{equation}\label{first-order-exp}
    (-\Delta)^s u(x) = u(x) + s L_{\Delta} u(x) + o(s) \quad s \to 0^+ \quad \text{in} \ L^p(\mathbb{R}^N), \ 1 < p \leq \infty.
\end{equation}  
Recall that for $s \in (0,1)$, the fractional Laplacian $(-\Delta)^s$ can be written as a singular integral operator defined in the principal value sense (see \cite[Section 3]{valdinoci})
\[
\begin{split}
(-\Delta)^su(x) = c(N,s) \ \text{P.V.} \int_{\R^N} \frac{u(x)-u(y)}{\abs{x-y}^{N+2s}} ~dx,
\end{split}
\] \\
where $c(N,s) =2^{2s} \pi^\frac{-N}{2} s \frac{\Gamma\l(\frac{N+2s}{2}\r)}{\Gamma(1-s)}$ is a normalizing constant. In the same spirit, the operator $L_\Delta$ has the following integral representation (see \cite[Theorem 1.1]{Chen-Weth-2019}):
\begin{equation}\label{def:log-lap-ope}
    L_\Delta u(x) = c_N \int_{\mathcal{B}_1(x)} \frac{u(x)-u(y)}{\abs{x-y}^N} ~dy -c_N\int_{\R^N\setminus\mathcal{B}_1(x)} \frac{u(y)}{\abs{x-y}^N} ~dy + \rho_Nu(x),
\end{equation}
where $\mathcal{B}_1(x) \subset \mathbb{R}^N$ denotes the Euclidean ball of radius $1$ centered at $x$ and
\[
c_N := \pi ^{\frac{-N}{2}}\Gamma(\frac{N}{2}), \quad \rho_N := 2\ln 2+ \psi(\frac{N}{2})-\gamma, \quad  \gamma := -\Gamma^{'}(1),\]
 $\gamma$ being the Euler-Mascheroni constant and $\psi := \frac{\Gamma^{'}}{\Gamma}$, the digamma function. The singular kernel in \eqref{def:log-lap-ope} is sometimes called of {\it zero-order} because it is the limiting case of hyper singular integrals. These types of operators are related to the geometric stable L\'evy process, we refer to \cite{Sikic-Song-Vondracek-2006, Kassmann-Mimica-2017} and the references therein for an overview of associated applications. Coming to the latest advancements in the study of small order limits of non-local operators, the authors in \cite{Dyda-Jarohs-Sk-2025} have recently introduced a non-linear extension of the logarithmic Laplacian.

The first motivation to study the problem \eqref{prob:logarithmic-laplace} comes from exploring the asymptotic behavior of the solution of the following weighted fractional Dirichlet problem:
\begin{align}
	\label{Eq:Problem}
    \tag{$\mathcal{P}_s$}
	{(-\Delta)^s} u  = a(s,x)\abs{u}^{p(s)-2}u 
        \quad \text{in } \Omega, \quad u= 0 \quad \text{in } \R^N\setminus \Omega,
\end{align}
where the exponent function $p: (0, \frac{1}{4}) \to \mathbb{R}^+$ satisfies
\begin{equation}\label{regu:p}
    p \in C^1([0, \frac{1}{4}]), \quad \lim_{s \to 0^+} p(s) =2
\end{equation}
and $a$ satisfies
\begin{enumerate}[label=\textnormal{(a$_0$)},ref=\textnormal{a$_0$}]
	\item\label{a:regulrity-asym} $  a(\cdot,x) \in C^1([0, \frac{1}{4}]) \quad \text{and} \quad   \lim_{s \to 0^+} a(s, x) = 1 \quad \text{for a.e.} \ x \in \Omega.$
	\end{enumerate}
 
Note that by passing limit $s\to 0^+$ on both sides of equation \eqref{Eq:Problem} (at least heuristically), in light of \eqref{regu:p} and \eqref{a:regulrity-asym}, we get \[
(-\Delta)^s u \to u \quad \text{and} \quad a(s,x)\abs{u}^{p(s)-2} u \to u.\]
Since this gives no information on the limiting profile of problem \eqref{Eq:Problem}, we have to rely on the first-order expansion (in $s$) on both sides of \eqref{Eq:Problem}, which naturally leads us to the logarithmic Laplace operator $L_\Delta.$
Thus,   
 \[
\begin{split}
 L_\Delta u :&= \lim_{s \to 0^+} \frac{d}{ds} (-\Delta)^s u = \lim_{s\to0^+} \frac{(-\Delta)^s u- \lim_{s\to0^+}(-\Delta)^su}{s}
\end{split}
\]
and 
\[
\begin{split}
\lim_{s\to0^+}\frac{ a(s,x)|u|^{p(s)-2}u- u}{s}  & = \lim_{s\to 0^+}\frac{\left(a(s,x)- 1\right) |u|^{p(s)-2}u_s}{s} + \lim_{s\to0^+}\frac{|u|^{p(s)-2} u- u}{s}\\
& = a'(0,x)u +p'(0) \ln\abs{u} u.
\end{split}
\]
This brings us to the study of the following limiting problem 
\begin{align}
    \label{Eq:LimitingProblem}
    {L_\Delta}u  = a'(0,x)u+p'(0)\ln\abs{u}u  \quad 
    \text{in } \Omega, \quad u= 0 \quad \text{in } \R^N\setminus\Omega
\end{align} 
which coincides with the problem \eqref{prob:logarithmic-laplace} when $f(x,u) = a'(0,x) u$ and $\sigma = p'(0).$ \vspace{0.1cm}\\ 
In the seminal work \cite{Chen-Weth-2019}, the authors have studied an eigenvalue problem involving the logarithmic Laplacian, {\it i.e.,} $f(x,u) = \lambda u$ and $\sigma=0$. They showed that the first Dirichlet eigenvalue $\lambda_{1, L}$ of $L_{\Delta}$ is  given by
\[
\lambda_{1, L}(\Omega) := \frac{d}{ds}\bigg|_{s=0} \lambda_{1,s}(\Omega),
\]
where $\lambda_{1,s}(\Omega)$ is the first Dirichlet eigenvalue of $(-\Delta )^s$ in $\Omega$ which is principal and simple. Moreover, the $L^2$-normalized Dirichlet eigenfunction of $(-\Delta )^s$ corresponding to $ \lambda_{1,s}(\Omega)$ converges to the $L^2$-normalized Dirichlet eigenfunction of $L_\Delta$ corresponding to $\lambda_{1,L}(\Omega)$ in $L^2(\Omega).$  Under additional conditions, they also proved the maximum principles in weak and strong forms and a Faber-Krahn type inequality. For more results on the bounds of eigenvalues and properties of eigenfunctions, we refer \cite{Laptev-Weth-2021, Feulefack-Jarohs-Weth-2022}.

 In \cite{Santamaria-Saldana-2022}, the authors have studied the problem \eqref{Eq:Problem} in the superlinear and subcritical case with $a(s,x)\equiv 1$ and have shown that the sequence of solutions $u_s$ of the problem \eqref{Eq:Problem} converges to a solution of the limiting problem \eqref{Eq:LimitingProblem} with $a'(0,x) \equiv 0$ and $p'(0) \in (0, \frac{4}{N}).$ Note that the exponent $\frac{4}{N}$ appears in the first order expansion of the fractional critical exponent $2_s^\ast.$ The authors in \cite{Angeles-Saldana-2023} conducted a similar study focusing on problem \eqref{Eq:Problem} in the sublinear case and $a(s,x) \equiv 1$. In this case, the solution of the problem \eqref{Eq:Problem} converges to the unique (positive) solution of the limiting problem \eqref{Eq:LimitingProblem} with $a'(0,x) \equiv 0$ and $p'(0) \in (-\infty, 0).$ In \cite{Frank-Konig-Tang}, an equation similar to \eqref{prob:logarithmic-laplace} is considered
with a conformally invariant logarithmic operator and $\sigma = \frac{4}{N}$ on the sphere. Therein, all non-negative solutions of this equation have been classified, such solutions are extremals of a Sobolev logarithmic inequality on the sphere proved in \cite[Theorem 3]{Beckner-1995}. In \cite{Chen-Zhou-2024}, the authors have studied a critical semilinear problem \eqref{prob:logarithmic-laplace} with $f \equiv 0$ and $\sigma = \frac{4}{N}$ in $\mathbb{R}^N$, showing the existence of non-negative solutions along with their classification. For a discussion on the regularity properties and boundary behavior of weak solutions of an elliptic problem involving the logarithmic Laplacian, we refer the reader to \cite{Chen-Weth-2019, Lara-Saldana-2022, Santamaria-Rios-Saldana-2024}.

Another motivation to study the problem \eqref{prob:logarithmic-laplace} arises from the structure of the non-linearity involved in the equation. When $f(x,u)=\omega(x)u, \  \omega \in L^\infty(\Omega),\ \omega >0$  and $\sigma \in (-\infty, 0)$, the problem \eqref{prob:logarithmic-laplace} can be seen as the non-local counterpart of the stationary logistic type equation for zero-order operators. More precisely, when $0 < u \leq 1$, the non-linearity $f(\cdot, u)$ acts as a source term and when $u >1$, the non-linearity $f(\cdot, u)$ acts as an absorption term. On the other hand, when $f(x,u)=\omega(x)u, \  \omega \in L^\infty(\Omega),\ \omega >0$ and $\sigma >0$, the problem \eqref{prob:logarithmic-laplace} can be seen as the non-local counterpart of the Brezis-Nirenberg type problem for zero-order operators, due to the absence of compact embeddings of corresponding energy space into Orlicz type spaces governed by critical logarithmic growth non-linearities (see Theorems \ref{thm:embd:results} and \ref{thm:sharp:embd:results} and Remark \ref{compact:argu}). For the results available on logistic type equations and Brezis-Nirenberg type problem for the Laplacian and the fractional Laplacian, we refer the reader to \cite{Brezis-Nirenberg-1983, Caffarelli-Dipierro-Valdinoci-2017, Servadei-Valdinoci-2013, Servadei-Valdinoci-2015}.

Motivated from the above discussion and existing literature, we focus in the present paper on the existence, uniqueness and regularity results for the limiting problem \eqref{prob:logarithmic-laplace} depending upon the value of the parameter $\sigma$. To the best of our knowledge, the key novelty of this work lies in the presentation of the first existence results related to the Brezis-Nirenberg and logistic type problems involving the logarithmic Laplacian with  linear/superlinear perturbations of the critical logarithmic non-linearity. In frame of these goals, first we prove new sharp continuous and compact embedding of the corresponding energy space $\h(\Omega)$ in Orlicz spaces. The sharpness of embeddings is justified by constructing a suitable class of functions with non-trivial logarithmic scaling properties. These sharp embeddings are of independent interest and can be used to study several other aspects of a semilinear elliptic PDE governed by the logarithmic Laplacian like multiplicity results, non-existence results, {\it etc.}

As an application of the sharp embeddings established in this work, and drawing inspiration from the analysis conducted in \cite{Angeles-Saldana-2023, Santamaria-Saldana-2022}, we prove the existence of a non-trivial least energy solution for the problem \eqref{prob:logarithmic-laplace} involving a large class of non-linearities (see below Theorems \ref{Main-res-limitingprob} and \ref{them:existence-limiting}). This work also serves as an extension of the existence results proved in \cite{Angeles-Saldana-2023, Santamaria-Saldana-2022}, since it incorporates subcritical and superlinear perturbations of the critical logarithmic non-linearity, in view of sharp compact embeddings (see below Remark \ref{compact:argu}).

Next, to study the uniqueness of weak solution, we show a D\'iaz-Saa type inequality for the logarithmic Laplacian. This inequality is of independent interest and can be used for other types of equations as well, in particular, porous medium type equations, doubly non-linear equations (see \cite{Arora-Giacomoni-Warnault-2020}). This inequality also complements the uniqueness result of least energy solutions for the sublinear problem \eqref{prob:logarithmic-laplace} when $f \equiv 0$ and $\sigma <0$ studied in \cite{Angeles-Saldana-2023} to the uniqueness of non-negative weak solution of the problem \eqref{prob:logarithmic-laplace} and covers a large class of non-linearities $f$ (see below Theorem \ref{thm:regularity-uniqueness}). Motivated from \cite{Angeles-Saldana-2023, Santamaria-Saldana-2022}, we also analyze the asymptotic behavior of the solutions of non-local weighted fractional elliptic equations with superlinear and sublinear growth non-linearities which has not been dealt with previously. 
 
The reader is being referred to the following section for a detailed description of the main results, wherein we have not only highlighted the key difficulties faced while addressing the aforementioned issues but also provided a comparison with the findings already existing in the literature. \vspace{0.1cm}\\
\textbf{Outline of the paper:} The rest of the paper is organized as follows. In Section \ref{main-results}, we give the definition of the function spaces involved, the notion of weak solutions and the statements of our main results. In section \ref{Sec:embed}, we prove new sharp continuous and compact embeddings of the energy space into Orlicz type spaces. In section \ref{brezis-Nirenberg}, we study the existence of solutions for the Brezis-Nirenberg type problem \eqref{prob:logarithmic-laplace} when $\sigma \in (0, \frac{4}{N})$ and analyze the asymptotic behavior of the solutions of the non-local weighted elliptic equation \eqref{Eq:Problem} in the superlinear case. In section \ref{logistic-results}, we study the existence and uniqueness results for a logistic type problem \eqref{prob:logarithmic-laplace} when $\sigma \in (-\infty,0)$ and analyze the asymptotic behavior of solutions of the non-local weighted elliptic equation \eqref{Eq:Problem} in the sublinear case. In this section, we prove the D\'iaz-Saa type inequality for the logarithmic Laplacian.

\section{Definitions and main results}\label{main-results}
In this section, we set up the functional framework in the energy spaces and give statements of the main results. 
Let $\Omega \subseteq \R^N$ be a bounded domain with Lipschitz boundary. For $q\in [1,\infty]$, we denote by $L^q(\Omega)$ the standard Lebesgue space with the norm
\[
\|u\|_{L^q(\Omega)}:= \left(\into |u|^q ~dx\right)^\frac{1}{q}\ \text{for} \ 1 \leq q < \infty \quad \text{and} \quad \|u\|_{L^\infty(\Omega)}:= \ess\sup_\Omega |u|.
 \]
 Define
\[
k, j: \mathbb{R}^N \setminus \{0\} \to \mathbb{R} \quad \text{as} \quad k(z) =  \frac{{\bf 1}_{B_1(z)}}{|z|^N} \quad \text{and} \quad j(z) =  \frac{{\bf 1}_{\mathbb{R}^N \setminus B_1(z)}}{|z|^N}.
\]
For the problem \eqref{prob:logarithmic-laplace}, the natural solution space $\h(\Omega)$ is defined as (see \cite{Chen-Weth-2019})
\[
\begin{split}
\h(\Omega) = \bigg\{u\in L^2(\Omega): & \ u=0 \  \text{in} \ \R^N \setminus \Omega \ \text{and}\\
& \intr \abs{u(x)-u(y)}^2 k(x-y) ~dx ~dy < +\infty\bigg\}.
\end{split}
\]
The inner product and norm on $\h(\Omega)$ are given, respectively, by
\[
\mathcal{E}(u,v) = \frac{c_N}{2}\intr (u(x)-u(y))(v(x)-v(y)) k(x-y) ~dx ~dy\]
and  
\[ \|u\|= (\mathcal{E}(u,u))^\frac{1}{2}.\]
Moreover, the quadratic form associated with the operator $L_\Delta$ is defined as
\begin{align}
\label{quadratic-form}
\mathcal{E}_L(u,v) = \mathcal{E}(u,v)-c_N \intr u(x)v(y) j(x-y) ~dx ~dy + \rho_N\int_{\R^N} uv ~dx.
\end{align}
Also,
  $$\lambda_{1,L} := \min\{{\mathcal{E}_L(u,u)}: u\in \h(\Omega), \|u\|_{L^2(\Omega)}=1\} \in \R \quad  \text{and} \quad  \lambda_{1,L} \leq \ln(\lambda_{1,s}),$$ where $\lambda_{1,s}$, $\lambda_{1,L}$ represent the first Dirichlet eigenvalue (in $\Omega$) of $(-\Delta)^s$  and $L_\Delta$ respectively. To the best of our knowledge, from \cite[Theorem 2.1]{Correa-DePablo-2018} or \cite[Corollary 2.3]{Laptev-Weth-2021} and \cite[Proposition 3.8]{Santamaria-Saldana-2022}), it is only known that
\[
\h(\Omega) \hookrightarrow L^2(\Omega) \ \text{is compact} 
\]
and for every $u \in \h(\Omega)$
\[
\begin{split}
\frac{4}{N} \into \ln (|u|)u^2 ~dx \leq \mathcal{E}_L(u,u)+\frac{4}{N}\ln (\|u\|_{L^2(\Omega)})\|u\|^2_{L^2(\Omega)} + a_N \|u\|^2_{L^2(\Omega)}
\end{split}
\]
(for more details see \cite[Proposition 3.8]{Santamaria-Saldana-2022}). Such type of inequality is known as Pitt’s inequality in the literature, which was first proposed and proved by Beckner in \cite[Theorem 3]{Beckner-1995} for the Schwarz function and was later extended in \cite[Proposition 3.8]{Santamaria-Saldana-2022} for functions in the energy space $\h(\Omega).$

The above Pitt's type inequality prompts an investigation for continuous and compact embedding of $\h(\Omega)$ into Orlicz spaces. However, due to the sign-changing property of the growth function $f(t)=t^2 \ln t$ in $[0, \infty)$, it cannot serve as a candidate for the generalized $\Phi$-function within the Orlicz space theory. To address this issue, it is natural to consider $t^2 \ln(e+t)$ as a generalized $\Phi$-function, in view of the following relation 
\begin{equation}\label{eq:relation:embd}
    t^2 \ln(e+t) = t^2 \ln(t) + t^2 \ln\l(1+\frac{e}{t}\r) \quad \text{for} \ t>0.
\end{equation}
Before giving the first result, we recall the definitions of generalized $\Phi$-function and Orlicz spaces.
\begin{definition}
	A function $\ph \colon [0,+\infty) \to [0,+\infty]$ is said to be a generalized $\Phi$-function, if $\ph$ is increasing and satisfies 
    \[
    \ph(0)=0, \quad \lim_{t\to 0^+} \ph(t) = 0, \quad \text{and} \quad \lim_{t \to +\infty} \ph(t) = +\infty.
    \]
Moreover, $\ph$ is said to be a generalized convex $\Phi$-function (denoted by $\ph \in \Phi_c$) if $\ph$ is left-continuous and convex.
    \end{definition}

Let $M(\Omega)$ be the set of all real valued Lebesgue measurable functions defined on $\Omega$ and  recall the following inequality:
\begin{equation}
\label{logineq}
\frac{t}{t+1}\leq \ln(1+t) \leq t \quad \text{for all} \ t>-1.
\end{equation} 
Let $\ph \colon [0,+\infty) \to [0,+\infty]$ be a generalized convex $\Phi$-function. The modular associated with $\ph$ is defined as
	\begin{align}
		\varrho_\ph (u) = \into \ph(\abs{u(x)}) ~dx.
	\end{align}
	The set
	\begin{align*}
		L^\ph (\Omega) := \{ u \in M(\Omega)\,:\, \varrho_\ph (\lambda u) < \infty \text{ for some } \lambda > 0 \}
	\end{align*}
	equipped with the Luxemburg norm
	\[
		\|u\|_\ph = \inf \bigg\{ \lambda > 0 \,:\, \varrho_\ph \l( \frac{u}{\lambda} \r)  \leq 1 \bigg\}
\]
is a Banach space (see \cite[Theorem 3.3.7]{Harjulehto-Hasto-2019}). \\
The first result corresponds to the new continuous and compact embeddings of the energy space $\h(\Omega)$ in Orlicz type spaces. 
\begin{theorem}\label{thm:embd:results} 
Let $\Omega\subseteq \R^N$ be a bounded domain and $\ph, \psi:[0,\infty) \to [0,\infty)$ be generalized convex $\Phi$-functions such that $\ph(t):= t^2 \ln(e+t)$ for all $t \geq 0$, 
\begin{enumerate}[label=\textnormal{($\psi_1$)},ref=\textnormal{$\psi_1$}]
       \item \label{compact:cond} {$\lim_{t\to 0^+} \frac{\psi(t)}{t} $ exists in $\mathbb{R}$} and $\lim_{t \to \infty} \frac{\psi(t)}{\ph(t)} = 0.$
   \end{enumerate}
Then, 
    \begin{enumerate}
        \item[\textnormal{(i)}] the embedding $\h(\Omega)\hookrightarrow L^\ph(\Omega)$ is continuous. 
        \item[\textnormal{(ii)}] the embedding $\h(\Omega)\hookrightarrow L^\psi(\Omega)$ is compact.
    \end{enumerate}
\end{theorem}
\begin{remark}
    The compactness result stated in the above theorem is a significant improvement compared to the results available in the literature (see \cite[Theorem 2.1]{Correa-DePablo-2018} or \cite[Corollary 2.3]{Laptev-Weth-2021} where compactness in $L^2(\Omega)$ is proved). We present some examples of the function $\psi$ which satisfy \eqref{compact:cond}:
\begin{enumerate}
    \item[\textnormal{(i)}] $\psi(t) = t^2 \ln^\theta(e+ |t|)$ for $\theta \in [0,1),$  
    \item[\textnormal{(ii)}] $\psi(t) = t^2 \ln(e+ \ln(1+|t|)).$
\end{enumerate}
\end{remark}
Next, we show the sharpness of the results obtained in Theorem \ref{thm:embd:results} by constructing a suitable class of functions in $\h(\Omega)$ with non-trivial logarithmic scaling parameters. 
\begin{theorem}\label{thm:sharp:embd:results}
    Let $\Omega\subseteq \R^N$ be a bounded, convex domain and $\ph, \gamma \colon [0,\infty) \to [0,\infty)$ be generalized convex $\Phi$-functions, $\ph(t):= t^2 \ln(e+t)$ for all $t \geq 0$ and there exists a $\beta >0$ such that
    \begin{equation}\label{embd:cond:smaller}
        \lim_{\ell \to \infty} \frac{\gamma(\ell \beta)}{\ph(\ell \beta)} =\infty.
    \end{equation}
 Then,
        \begin{enumerate}
        \item[\textnormal{(i)}] the embedding $\h(\Omega) \hookrightarrow L^\ph(\Omega)$ is not compact and 
        \item[\textnormal{(ii)}] the embedding $\h(\Omega) \hookrightarrow L^\gamma(\Omega)$ is not continuous. 
    \end{enumerate}
    
\end{theorem}
\begin{remark}
    The above result indicates that if any Orlicz space $L^\gamma(\Omega)$ is smaller than $L^\ph(\Omega)$ in the sense of \eqref{embd:cond:smaller}, {\it i.e}, $L^\gamma(\Omega) \hookrightarrow L^\ph(\Omega)$, then the embedding $\h(\Omega) \hookrightarrow L^\gamma(\Omega)$ is not continuous. Moreover, if \eqref{embd:cond:smaller} holds for some $\beta>0$, then it holds for all $\beta'>0.$
\end{remark}
\begin{remark}\label{compact:argu}
In view of inequalities \eqref{eq:relation:embd} and \eqref{logineq}, the function $\ph$ can be rewritten as
\[
\ph(|t|)=  t^2 \ln(|t|) + t^2 \ln\l(1+\frac{e}{|t|}\r),
\]
where the second term can be estimated as
    \[
    t^2 \ln\l(1+\frac{e}{|t|}\r) \leq e t \quad \text{for} \ t \in \mathbb{R}.
    \]
Given the compact embedding $\h(\Omega) \hookrightarrow L^1(\Omega)$, it follows that only the first term $t^2 \ln(t)$ contributes to the failure of compactness in the embedding result stated in Theorem \ref{thm:sharp:embd:results}, which motivates the study of the problem \eqref{prob:logarithmic-laplace}.
\end{remark}
Next, we introduce the functional framework and the notion of weak solution in order to study our main problem \eqref{prob:logarithmic-laplace}. The energy functional $\mathbb{E}$: $\h(\Omega)$ $\longrightarrow \R $ associated to \eqref{prob:logarithmic-laplace} is given by
\begin{equation}\label{def:enerfunct:plambda}
    \mathbb{E}(u)= \frac{1}{2}\mathcal{E}_L(u,u)- \into F(x,u) ~dx -\frac{\sigma}{4}\into u^2(\ln(|u|^2)-1) ~dx,
\end{equation}
where $F(x,t) = \int_0^t f(x,z) ~dz$ is the primitive of $f$ and $f: \Omega\times\R \to \R$ is a function satisfying the superlinear growth assumptions:
\begin{enumerate}
[label=\textnormal{($f_1$)}, ref=\textnormal{$f_1$}]
\item  \label{assump f1} 
$f \in L^\infty(\Omega \times K)$ for any $K \subset \mathbb{R}$, $|K| < +\infty,$
\end{enumerate}

\begin{enumerate}
[label=\textnormal{($f_2$)}, ref=\textnormal{$f_2$}]
\item  \label{assump f2} 
$\lim_{|t| \to \infty} \frac{f(x,t)}{t\ln|t|}=0$ uniformly in $\Omega,$
\end{enumerate}
\begin{enumerate}
[label=\textnormal{($f_3$)}, ref=\textnormal{$f_3$}]
\item  \label{assump f3} 
there exists $g \in L^\infty(\Omega)$ such that $\lim\limits_{t \to 0} \frac{f(x,t)}{t} = g(x)$ uniformly in $\Omega.$
\end{enumerate}

In light of the logarithmic Sobolev inequality (see \cite[Proposition 3.8]{Santamaria-Saldana-2022}), Theorem \ref{thm:embd:results} (i), Remark \ref{compact:argu}, \eqref{assump f1} and \eqref{assump f2}, the energy functional $\mathbb{E}$ is well defined in $\mathbb{H}(\Omega).$ Next, we introduce the notion of a weak solution to the problem \eqref{prob:logarithmic-laplace}.
\begin{definition}
     A function $u \in \h(\Omega)$ is said to be a ``weak solution" of \eqref{prob:logarithmic-laplace} if
 \[\mathcal{E}_L(u,v) = \into (f(x,u) + \sigma u \ln\abs{u}) v  ~dx \quad  \text{for  every} \ v\in \h(\Omega).\]
    \end{definition}
    
The set of all non-trivial solutions of \eqref{prob:logarithmic-laplace} belongs to the set
\[
N_{0, \sigma}:= \{u\in \h(\Omega)\setminus\{0\}: \mathcal{E}_L(u,u)=\into ( f(x,u) + \sigma u \ln\abs{u})u ~dx\}, 
\]
which is the  Nehari manifold associated with \eqref{prob:logarithmic-laplace}. Moreover, a function $u\in N_{0, \sigma}$ is called a least energy solution of \eqref{prob:logarithmic-laplace} if 
\[
\mathbb{E}(u) = \inf_{v\in N_{0,\sigma}} \mathbb{E}(v).
\]
The first result corresponds to the existence of a weak solution for the problem \eqref{prob:logarithmic-laplace} when $\sigma \in (0, \frac{4}{N})$. 
\begin{theorem}
    \label{Main-res-limitingprob} 
 For every $\sigma \in (0,\frac{4}{N})$ and $f$ satisfying \eqref{assump f1}-\eqref{assump f3}, 
 \begin{enumerate}
[label=\textnormal{($f_4$)}, ref=\textnormal{$f_4$}]
\item  \label{assump f4} 
$f(x,\cdot)\in C^1(\R)$ such that $\lim_{|t| \to \infty} \frac{f'(x,t)}{\ln|t|}=0$ uniformly in $\Omega,$
\end{enumerate}
 \begin{enumerate}
[label=\textnormal{($f_\sigma$)}, ref=\textnormal{$f_\sigma$}]
\item  \label{assump f-sigma} there exists a constant $\delta < \sigma$ such that
\[
(t^2f'(x,t)+\delta t^2 -tf(x,t)) \geq 0 \quad \text{for all} \ t \in \mathbb{R} \ \text{and a.e.} \ x \in \Omega,
\]
\end{enumerate}
the problem \eqref{prob:logarithmic-laplace} has a least energy solution $u\in \h(\Omega)\setminus\{0\}.$ Furthermore, if $F(x,t) \leq F(x, |t|)$ for all $t \in \mathbb{R}$ and for a.e. $x \in \Omega$, all least energy solutions of \eqref{prob:logarithmic-laplace} do not change sign in $\Omega$.
\end{theorem}
\begin{remark}\label{remark-AR-cond}
Note that the condition \eqref{assump f-sigma} implies
\begin{equation}\label{modi-AR-cond}
    G(x,t):=t f(x,t) - 2 F(x,t) + \frac{\delta t^2}{2} \geq 0 \quad \text{for all} \ t \in \mathbb{R} \ \text{and a.e. in} \ \Omega.
\end{equation}
The condition \eqref{assump f-sigma} framed for the study of problem \eqref{prob:logarithmic-laplace} can be seen as an alternative to the classical Ambrosetti-Rabinowitz condition used in the study of superlinear boundary value problem involving the Laplacian and the fractional Laplacian. Moreover, \eqref{assump f-sigma} holds if the map
\[
t \mapsto G(x,t) \ \text{is increasing for} \ t \geq 0 \ \text{and decreasing for} \ t \leq 0.
\]
\end{remark}
\begin{remark}\label{examples}
The existence result in Theorem \ref{Main-res-limitingprob} addresses a wide range of non-linearities $f(x, \cdot)$ that adhere to the superlinear and subcritical growth conditions outlined in \eqref{assump f1}-\eqref{assump f4} and \eqref{assump f-sigma}. Some of the examples (motivated from the embedding in Theorem \ref{thm:embd:results}) satisfying assumptions \eqref{assump f1}- \eqref{assump f4} and \eqref{assump f-sigma} are:
\begin{itemize}
    \item[\textnormal{(i)}] $f(x,t) = g(x) t \ln^\theta(e+|t|)$ for any $\theta \in [0,1)$ and $g \in L^\infty(\Omega),$
    \item[\textnormal{(ii)}] $f(x,t) = g(x)t \ln(\mu+ \ln(1+|t|))$ for any $\mu >0$ and $g \in L^\infty(\Omega).$
\end{itemize}
\end{remark}
\begin{remark}
The condition $F(x,t) \leq F(x, |t|)$ for all $t \in \mathbb{R}$ and for a.e. $x \in \Omega$ in Theorem \ref{Main-res-limitingprob} holds with equality if $f(x, \cdot)$ is an odd function. In particular, it holds for the examples provided in Remark \ref{examples}. 
\end{remark}

Note that in view of \eqref{eq:relation:embd}, the energy functional $\mathbb{E}$ can be rewritten as
\begin{equation}\label{energy:func:pertu}
    \begin{split}
\mathbb{E}(u) &= \frac{1}{2}\mathcal{E}_L(u,u) + \frac{1}{2}\into \l(\frac{\sigma u^2}{2} - 2F(x,u)\r) ~dx - \frac{\sigma}{2}\into \ph(|u|)~dx\\
& \qquad \qquad \qquad + \frac{\sigma }{2}\into u^2 \ln\l(1+\frac{e}{|u|}\r)  ~dx
\end{split}
\end{equation}
where the last term in the above energy functional \eqref{energy:func:pertu} is well defined because of \eqref{logineq}.

Due to the presence of the modular term $\into \ph(|u|)~dx$ and the lack of compactness of the embedding $\h(\Omega) \hookrightarrow L^\ph(\Omega)$, the functional $\mathbb{E}$ in \eqref{energy:func:pertu} does not satisfy the Palais-Smale condition. In the seminal studies of the Brezis-Nirenberg problem, \cite{Brezis-Nirenberg-1983} addresses the Laplacian and \cite{Servadei-Valdinoci-2015} tackles the fractional Laplacian, where the existence of a positive solution is demonstrated by using the Ambrosetti-Rabinowitz mountain pass theorem-a geometric argument-without relying on the Palais-Smale condition (see \cite[Theorem 2.2]{Brezis-Nirenberg-1983}). Although the same geometric argument is applicable to our problem \eqref{prob:logarithmic-laplace}, the primary challenge stems from employing the conclusions drawn from these geometric arguments, specifically in demonstrating the non-triviality of the solution obtained from the limit of the minimizing sequence. Consequently, significant obstacles arise in applying standard variational methods to identify the critical points of $\mathbb{E}$.

To prove the above existence result, the main variational tool applied is the method of Nehari manifold. However, the implementation of this approach encounters significant difficulties. The first and foremost challenge is that the convergence of the minimizing sequence cannot be established through the known compact embedding of the energy space $\h(\Omega)$ into $L^2(\Omega)$. Additionally, to gain the loss of compactness, we use the sharp logarithmic Sobolev-type inequality (see \cite[Proposition 3.8]{Santamaria-Saldana-2022}) and Theorem \ref{thm:embd:results} , which plays a major role in our analysis. Together with this and exploiting the assumption $\sigma \in (0, \frac{4}{N})$, we gain the compactness of the minimizing sequence in the Nehari manifold (see Proposition \ref{weak-conve-est}).
To show the non-triviality of the solution, we perform the analysis of the fibering map and show that the associated Nehari manifold is away from the origin (see Lemma \ref{Nehari-struc} and Lemma \ref{lower-bds-norms}). The embeddings in Theorem \ref{thm:embd:results} also play an important role in deriving the limit estimates for non-linear terms with superlinear perturbation (see Lemma \ref{limits_interchange}). \\
Next we study the problem \eqref{prob:logarithmic-laplace} when $\sigma \in (-\infty, 0)$, which corresponds to the ``sublinear regime". We first state the existence result for the least energy solution of the problem \eqref{prob:logarithmic-laplace} in the sublinear regime.
\begin{theorem}
      \label{them:existence-limiting}
      Let $f$ satisfies \eqref{assump f1}-\eqref{assump f3} and $\sigma \in (-\infty,0)$.  Then, the problem \eqref{prob:logarithmic-laplace} has a non-trivial, least energy solution in $\h(\Omega).$ Furthermore, if $F(x,t)\leq F(x,|t|)$ for all $t \in \R$, all least energy solutions of \eqref{prob:logarithmic-laplace} do not change sign in $\Omega$.
  \end{theorem}  
  The analysis in the sublinear regime differs from that in the superlinear regime. We observe that the functional $\mathbb{E}$ is coercive on the energy space in the sublinear regime, thus, the existence of a solution can be proved for all $\sigma \in (-\infty,0)$, which contrasts with the superlinear regime, where the existence of solutions is proved for all $\sigma \in (0,\frac{4}{N})$ on the Nehari manifold.
  
  Next, we address the uniqueness of non-trivial, non-negative solution of  sublinear problem \eqref{prob:logarithmic-laplace} with the help of D\' iaz-Saa type inequality for the logarithmic Laplacian, which was also an open problem. Such an inequality is also related to Picone-type identity, majorly used to study the simplicity of the eigenvalues, Sturmian comparison principles, oscillation theorems, Hardy and Barta type inequalities. We refer the reader to \cite{Allegretto-Huang-1998, Arora-Giacomoni-Warnault-2020} and the references therein for more applications of such types of inequalities. 

Our next result pertains to the (global) D\' iaz-Saa type inequality for the logarithmic Laplacian.
\begin{theorem}\label{thm:modi-Diaz-Saa-inequ}
Let $w_1, w_2 \in \h(\Omega)$ such that $w_1, w_2>0$ a.e. in $\Omega$ and $\frac{w_1}{w_2}, \frac{w_2}{w_1} \in L^\infty(\Omega)$. Then, for $q=2$, the following inequality holds:
\begin{equation}\label{Diaz-Saa-ineq}
    \mathcal{E}_L\l(w_2, \frac{w_2^q-w_1^q}{w_2^{q-1}}\r) - \mathcal{E}_L\l(w_1, \frac{w_2^q-w_1^q}{w_1^{q-1}}\r) \geq 0.
\end{equation}
Moreover, if  $h_\Omega(x) + \rho_N \geq 0$ in $\Omega$, then the above inequality holds true for all $q \in [1,2).$
\end{theorem}
\begin{remark}
    Note that the condition $h_\Omega + \rho_N \geq 0$ in $\Omega$ in Theorem \ref{thm:modi-Diaz-Saa-inequ} required for the case $q \in [1,2)$ holds true when the domain $\Omega$ is small enough (see \cite[Lemma 4.11]{Chen-Weth-2019}). This condition also appears in the validity of the maximum principle and the positivity of first eigenvalue for the logarithmic Laplacian.
\end{remark}

The proof of the above inequality consists mainly of two steps. First, notice that the inequality \eqref{Diaz-Saa-ineq} is equivalent to the convexity of the following functional 
    \[
    W_q(u) =\mathcal{E}_L(u^{\frac{1}{q}},u^{\frac{1}{q}})
    \]
in the convex cone $V_+^q = \{u : \Omega \to (0,\infty) \ | \ u^{\frac{1}{q}} \in \h(\Omega)\}.$ In the case of the fractional Laplacian, the convexity of the corresponding energy functional can be established through straightforward algebraic inequalities (see \cite[Proposition 4.2]{Brasco-Franzina-2014}). However, this is not true for the operator \( L_\Delta \) due to the sign-changing nature of the bilinear form \( \mathcal{E}_L(\cdot, \cdot) \). In order to overcome this problem, we exploit the equivalent representation of the bilinear form \( \mathcal{E}_L(\cdot, \cdot) \) from \cite[Proposition 3.2]{Chen-Weth-2019}. Secondly, the convexity of the restriction map \( W_q: V_+^q \to \mathbb{R}^+ \) when applied to \( V_+^q \) is well-known to be equivalent to the monotonicity of its subdifferential \( \partial W_q(u) \) at \( u \in V_+^q \). This subdifferential is a non-empty set only for certain elements \( u \in V_+^q \), and determining these elements can be quite challenging. To avoid this issue, we restrict ourselves only to certain directional derivatives of $W_q$ which exist in the classical sense. The monotonicity of such directional derivative maps leads to the above D\'iaz-Saa type inequality. \\ 
As an application of the above D\'iaz-Saa type inequality, we state the following uniqueness and regularity result for the solution of the problem \eqref{prob:logarithmic-laplace} in the sublinear regime.
  \begin{theorem}
      \label{thm:regularity-uniqueness}
  Let $\sigma \in (-\infty,0)$, $f$ satisfies \eqref{assump f1}, \eqref{assump f2}, 
  and $u \in \h(\Omega)$ be a non-trivial non-negative weak solution of \eqref{prob:logarithmic-laplace}.
     Then, 
     \begin{enumerate}
         \item[\textnormal{(i)}] $u \in L^\infty(\Omega)$ and $u \in C(\R^N).$
  \item[\textnormal{(ii)}] If there exists a $\theta>0$ such that
  \begin{equation}\label{f-est-near-0}
       f(x,t) + \sigma t \ln |t| \geq 0 \quad \text{for all} \ t \in [0, \theta),
  \end{equation}
  then $u>0$ in $\Omega$ and there is a $C>0$ such that
   \begin{equation}
   \label{u_reg-sublinear}
   C^{-1} \ell^\frac{1}{2}(\delta(x)) \leq u(x) \leq C \ell^\frac{1}{2}(\delta(x)) \quad \text{for all} \ x \in \Omega,
   \end{equation}
where \[\ell(r):= - \frac{1}{\ln (\min \{r,\frac{1}{10}\})}.\] 
\item[\textnormal{(iii)}] $u$ is unique if $f$ satisfies \eqref{f-est-near-0} and the map
\begin{equation}\label{f-decr-prop}
    t \mapsto \frac{f(x,t) + \sigma t \ln |t|}{t^{q-1}}\ \text{is decreasing in} \ \R^+ \ \text{for} \ q \in [1,2]
\end{equation}
and $h_\Omega(x) + \rho_N \geq 0$ in $\Omega$ if $q \in [1,2)$.
     \end{enumerate}
     \end{theorem}
\begin{remark}
    Note that both the examples in Remark \ref{examples} with $\mu >1$ and $g \equiv 1$ satisfy \eqref{f-est-near-0} for any $\theta>0$ and satisfy \eqref{f-decr-prop} with $q=2.$
\end{remark}
\begin{remark}
    The uniqueness result in Theorem \ref{thm:regularity-uniqueness} (iii) corresponds to any non-negative weak solution of the problem \eqref{prob:logarithmic-laplace} with $\sigma \in (-\infty,0)$ as compared to the uniqueness result of least energy weak solution of the problem \eqref{prob:logarithmic-laplace} with $f \equiv 0$ in \cite[Theorem 6]{Angeles-Saldana-2023}. Moreover, in Theorem \ref{thm:regularity-uniqueness} (iii), we cover a large class of non-linearities $f$ satisfying \eqref{assump f1}, \eqref{assump f2}, \eqref{f-est-near-0} and \eqref{f-decr-prop}.
\end{remark}    
To study the asymptotic behavior of the solution of the problem \eqref{Eq:Problem}, we introduce the solution space and the corresponding energy functional. The natural solution space $H^s_0(\Omega)$, $s \in (0,1)$, is defined as
\[
\begin{split}
H^s_0(\Omega) = \{u\in H^s(\R^N): u=0  \ \text{in} \ \R^N\setminus \Omega\},
\end{split}
\]
where
\[
\begin{split}
H^s(\Omega) = \{u\in L^2(\Omega):  \frac{\abs{u(x)-u(y)}}{\abs{x-y}^{\frac{N}{2}+s}} \in L^2(\Omega \times \Omega)\}
\end{split}
\]
is the usual fractional Sobolev space defined in \cite{valdinoci}. The inner product and norm on $H^s_0(\Omega)$ are given, respectively, by
    \[\mathcal{E}_s(u,v) = \frac{c_{N,s}}{2}\intr \frac{(u(x)-u(y))(v(x)-v(y))}{\abs{x-y}^{N+2s}} ~dx ~dy\]
and  
\begin{equation}\label{Hs-norm}
    \|u\|_s=  [\mathcal{E}_s(u,u)]^\frac{1}{2}.
\end{equation}
The energy functional $E_s$: $H^s_0(\Omega) \to \R $ associated to \eqref{Eq:Problem} is given by
\[
E_s(u)= \frac{1}{2}\|u\|_s^2- \frac{1}{p(s)}\into a(s,x)\abs{u}^{p(s)} ~dx.\]
\begin{definition}
    A function $u_s \in H^s_0(\Omega)$ is said to be a ``weak solution" of \eqref{Eq:Problem} if
 \[
 \mathcal{E}_s(u_s,v) = \into a(s,x)\abs{u_s}^{p(s)-2}u_s v ~dx \quad 
    \text{for every} \ v\in H^s_0(\Omega).
\]
\end{definition}

The Nehari manifold consisting of all non-trivial solutions of \eqref{Eq:Problem} is given by 
\[
N_s = \{u\in H^s_0(\Omega)\setminus\{0\}: \|u\|^2_s=\into a(s,x)\abs{u}^{p(s)} ~dx\}.
\]
Moreover, a function $u\in N_s$ is called least energy solution of \eqref{Eq:Problem} if 
\[E_s(u)= \inf_{v\in N_s} E_s(v).\]
To study the asymptotics of the problem \eqref{Eq:Problem} with $p$ satisfying \eqref{regu:p} and sublinear or superlinear type non-linearity, we assume the following conditions on the function $a(\cdot,\cdot)$:
\begin{enumerate}[label=\textnormal{(a$_1$)},ref=\textnormal{a$_1$}]
	\item\label{weifunc:bound} There exist constants $c_1, c_2 >0$ (independent of $s$) such that 
 \[c_1 \leq \|a(s,\cdot)\|_{L^{\beta(s)}(\Omega)}^{\frac{1}{p(s)-2}} \leq  c_2 \quad \text{for all} \ s \in (0,\frac{1}{4}),\]
where $\beta:(0,\frac{1}{4}) \to \R$ be the exponent function such that 
       \begin{equation}
     \label{assumption on beta}
        \beta(s)> \max\l\{\frac{2_s^\ast}{2_s^\ast-p(s)}, 1+\frac{(N-2s)}{2s(\delta-\gamma)}\r\} \quad \text{for all $s \in (0, \frac{1}{4}),$}
     \end{equation} 
     for some $\gamma \in (0,\delta)$ and $\delta=1-\frac{Np'(0)}{4}.$
	\end{enumerate}
\begin{enumerate}[label=\textnormal{(a$_2)$},ref=\textnormal{a$_2$}]
	\item\label{a:regulrity-1-asym}  $\partial_s a(0,\cdot) \in L^\infty(\Omega).$
	\end{enumerate}
 \begin{enumerate}[label=\textnormal{(a$_3$)},ref=\textnormal{a$_3$}]
	\item\label{cond:a-upperbound} There exists a $c_3>0$ (independent of $s$) such that
  \[\|a(s,\cdot)\|_{L^\infty(\Omega)}^{\frac{1}{s}} \leq c_3 \quad \text{for all} \ s \in (0,\frac{1}{4}).\]
	\end{enumerate}
The upcoming two results correspond to the asymptotic behavior of solution of the non-local problem \eqref{Eq:Problem} with superlinear and sublinear growth non-linearities.  
\begin{theorem}
    \label{thm:asym-super} 
Let $p: [0,\frac{1}{4}] \to \R^+$ be a $C^1$ function satisfying $p(s)\in (2, 2^*_s)$, $p'(0) \in (0,\frac{4}{N})$ and $a : [0, \frac{1}{4}] \times \Omega \to \mathbb{R}$ satisfies \eqref{a:regulrity-asym}, \eqref{weifunc:bound} and \eqref{a:regulrity-1-asym}.
Let $u_{s}\in H^{s}_0(\Omega)$ be a least energy solution of the problem \eqref{Eq:Problem}.
Then, there is a least energy solution $u_0 \in \h(\Omega)\setminus\{0\}$ of the problem \eqref{Eq:LimitingProblem}
such that, passing to a subsequence, \[u_{s}\to u_0 \in L^2(\R^N) \ \text{as}\ s \to 0^+.\]
Also, \[
\lim_{s \to 0^+} \frac{E_{s}(u_{s})}{s} = \mathbb{E}(u_0) = \frac{p'(0)}{4} \|u_0\|^2_{L^2(\Omega)} 
\]
and
\[\lim_{s \to 0^+} \|u_{s}\|_{s}= \|u_0\|_{L^2(\Omega)}.\]
\end{theorem}
\begin{theorem}
      \label{thm-asym-sub}
Let $p: [0,\frac{1}{4}] \to \R^+$ be a $C^1$ function satisfying $p(s)\in (1,2)$, $p'(0) \in (-\infty, 0)$ and $a$ satisfies \eqref{a:regulrity-asym}, \eqref{a:regulrity-1-asym} and \eqref{cond:a-upperbound}. Let $u_s$ be a positive least energy solution of \eqref{Eq:Problem}. Then, $u_s \to u_0$ in $L^q(\R^N)$ as $s \to 0^+$, \ for all $1\leq q< \infty$, where $u_0 \in \h(\Omega)\cap C(\R^N)$ is the unique, non-trivial, non-negative least energy solution of \eqref{Eq:LimitingProblem} satisfying \eqref{u_reg-sublinear}.
\end{theorem}
The study of the asymptotic behavior in both the superlinear and sublinear regimes follow different strategies. In the superlinear case, the proof of Theorem \ref{thm:asym-super} relies on the variational method, uniform energy-derived estimates, and the logarithmic Sobolev inequality (see Proposition \ref{Log-Sobolev-ineq}). In contrast, the proof of Theorem \ref{thm-asym-sub} in the sublinear case utilizes sharp regularity bounds. This approach also employs direct integral estimates to establish a uniform bound for the solutions of equation \eqref{Eq:Problem} in the norm of \(\h(\Omega)\). This bound together with the compact embedding \(\h(\Omega) \hookrightarrow L^2(\Omega)\) provides the main compactness argument needed to characterize the limiting profile. 
\section{Embeddings in Orlicz spaces}\label{Sec:embed}
In this section, we prove new sharp continuous and compact embeddings of the energy space $\h(\Omega)$ into Orlicz spaces. We begin with introducing some preliminaries of the Orlicz spaces and then we look for an Orlicz space $L^\psi(\Omega)$ that compactly embeds $\h(\Omega).$ In order to claim the sharpness of our embedding, we construct a bounded class of functions in $\h(\Omega)$ and $L^\ph(\Omega)$ with appropriate logarithmic scaling parameters, which does not have a convergent subsequence in $L^\ph(\Omega)$ and the same class is not bounded in $L^\gamma(\Omega) \hookrightarrow L^\ph(\Omega).$
\subsection{Preliminaries of Orlicz spaces}
\begin{definition}
	Let  $\ph \colon (0,+\infty) \to \R$ be a function and $p,q>0$. We say that
	\begin{enumerate}
		\item[\textnormal{(i)}]
			$\ph$ is almost increasing, if there exists $a \geq 1$ such that $\ph(s) \leq a \ph(t)$ for all $0 < s < t$.
		\item[\textnormal{(ii)}]
			$\ph$ is almost decreasing, if there exists $a \geq 1$ such that $a \ph(s) \geq \ph(t)$ for all $0 < s < t$.
	\end{enumerate}
We say that $\ph$ satisfies the property
	\begin{enumerate}
		\item[\textnormal{(aInc)}$_p$]
			if $\frac{\ph(t)}{t^p}$ is almost increasing.
		\item[\textnormal{(aDec)}$_q$]
			if $\frac{\ph(t)}{t^q}$ is almost decreasing.
\end{enumerate}
\end{definition}
\begin{definition}
	A function $\ph$ is said to be a generalized weak $\Phi$-function (denoted by $\ph \in \Phi_w$) if $\ph$ satisfies \textnormal{(aInc)}$_1$ in $(0, +\infty)$.
    \end{definition}
Next, we recall some important relations between the norm and the modular.
\begin{lemma}
	\label{norm_modular}
    \cite[Lemma 3.2.9]{Harjulehto-Hasto-2019}\newline
	Let $\ph \colon [0,+\infty) \to [0,+\infty]$ be a generalized weak $\Phi$-function that satisfies \textnormal{(aInc)}$_p$ and \textnormal{(aDec)}$_q$, $1 \leq p \leq q < \infty$. Then
	\begin{align*}
 \min \bigg\{ {\l(\frac{1}{a}\varrho_\ph (u)\r)}^{\frac{1}{p}} , {\l(\frac{1}{a}\varrho_\ph (u)\r)}^{\frac{1}{q}} \bigg\}
		\leq \|u\|_\ph
		\leq \max \bigg\{ {\l(a\varrho_\ph (u)\r)}^{\frac{1}{p}} , {\l(a\varrho_\ph (u)\r)}^{\frac{1}{q}} \bigg\},
	\end{align*}
	for all $u \in M(\Omega)$, where $a$ is the maximum of the constants of \textnormal{(aInc)}$_p$ and \textnormal{(aDec)}$_q$.
\end{lemma}

\begin{lemma}
	\label{epsilon}
    \cite[Lemma 3.1]{Arora-Crespo-Blanco-Winkert-2025} \newline
	The function $f_\eps \colon [0,+\infty) \to  [0,+\infty)$ given by
	\begin{align*}
		f_\eps(t) = \frac{ t^\eps } { \ln (e + t) }
	\end{align*}
	is almost increasing for $0 < \eps < \kappa :=e/(e + t_0)$ with constant $a_\eps :=\frac{f_{\eps}(t_{1, \eps})}{f_{\eps}(t_{2, \eps})} >1,$ where $t_0$ being the only positive solution of $t_0 = e \ln(e + t_0)$ and $t_{1,\eps}$, $t_{2, \eps}$ are respectively, the points of local maximum and minimum of $f_\eps$ in $(0, +\infty)$.   
\end{lemma}

    \begin{corollary}
	\label{ln(e+t)}
	Let $\ph:[0,\infty) \to [0,\infty)$ be given by
\[\ph(t)=t^2 \ln (e+t),\] then, it fulfills \textnormal{(aDec)}$_{2 + \eps}$, for $0 < \eps < \kappa$ and with constant $a_\eps$, where $\kappa$ and $a_\eps$ are the same as in Lemma \ref{epsilon}.
\end{corollary}
\subsection{Continuous and compact embeddings}
Let $\ph:[0,\infty) \to [0,\infty)$ be given by
\[\ph(t)=t^2 \ln (e+t).\]
Then, $\ph$ is a generalized convex $\Phi$-function and satisfies \textnormal{(aInc)}$_2$ and \textnormal{(aDec)}$_{2+\eps}$ for $0< \eps < \kappa,$ where $\kappa$ is defined in Lemma \ref{epsilon}.
For the above choice of $\ph$, the space
\[L^\ph(\Omega)= \bigg\{u\in M(\Omega) \,:\, \into |\lambda u|^2 \ln (e+|\lambda u|) ~dx < \infty \ \text{for some} \ \lambda>0\bigg\}\] equipped with the norm
\[\|u\|_\ph := \inf\bigg\{\lambda>0 \,:\, \into \abs{\frac{u}{\lambda}}^2\ln\l(e+\abs{\frac{u}{\lambda}}\r) ~dx \leq 1 \bigg\}\]
is a Banach space (see \cite[Theorem 3.3.7]{Harjulehto-Hasto-2019}). 
\begin{theorem}\label{thm:emb-cts}(Embeddings in Orlicz space)
The space $\h(\Omega)$ is continuously embedded in $L^\ph(\Omega).$ 
\end{theorem}
\begin{proof} 
To prove the claim, we show that the inclusion map $\mathcal{I}: \h(\Omega) \to L^\ph(\Omega)$ defined by $\mathcal{I}(u) = u$ maps bounded sets into bounded sets. Let $A$ be a bounded subset of $\mathbb{H}(\Omega)$. Therefore, there exists a constant $C_1>0$, such that
\begin{equation}\label{embd:est-0}
    \|u\|^2 = \mathcal{E}(u,u) \leq C_1 \quad \text{for all} \ u\in A.
\end{equation}
Then, in view of Lemma \ref{norm_modular}, it is enough to show that there exists a $C_2 >0 $ such that
\begin{equation}\label{embd:est-1}
    \into \ph(|u|) ~dx = \into |u|^2 \ln (e+|u|) ~dx \leq C_2 \quad  \text{for all} \ u \in A.
\end{equation} 
Let $u \in \h(\Omega).$ Using \eqref{eq:relation:embd}, \eqref{logineq}, Proposition \ref{Log-Sobolev-ineq}, and \cite[Lemma 3.4]{Santamaria-Saldana-2022}, we obtain 
\begin{equation}\label{embd:est-2}
    \begin{split}
   \into |u|^2 & \ln (e+|u|) ~dx \\
   &= \into |u|^2 \ln (|u|) ~dx + \into |u|^2 \ln \l(1+\frac{e}{|u|}\r) ~dx \\
   &\leq \frac{N}{4}\l(\mathcal{E}_L(u,u)+\frac{4}{N}\ln(\|u\|_{L^2(\Omega)})\|u\|_{L^2(\Omega)}^2 + a_N \|u\|_{L^2(\Omega)}^2\r) + e \|u\|_{L^1(\Omega)}\\
   & \leq \frac{N}{4}\mathcal{E}(u,u)-\frac{N c_N}{4} \intb \frac{u(x)u(y)}{|x-y|^N} ~dx ~dy\\
   & \qquad \qquad  +\frac{N (\rho_N + a_N)}{4} \|u\|_{L^2(\Omega)}^2 + \ln(\|u\|_{L^2(\Omega)}) \|u\|_{L^2(\Omega)}^2 + e |\Omega|^\frac{1}{2} \|u\|_{L^2(\Omega)}\\
   & \leq \frac{N}{4}\mathcal{E}(u,u) +\frac{N (\rho_N + a_N)}{4} \|u\|_{L^2(\Omega)}^2 + \ln(\|u\|_{L^2(\Omega)}) \|u\|_{L^2(\Omega)}^2 \\
   & \quad + e |\Omega|^\frac{1}{2} \|u\|_{L^2(\Omega)}. 
   \end{split}
\end{equation} 
Finally, by using \eqref{embd:est-0} and the compact embedding of $\h(\Omega)$ in $L^2(\Omega)$ in \eqref{embd:est-2}, we obtain the claim in \eqref{embd:est-1}.
\end{proof}
Next, we study the compactness of the energy space $\h(\Omega)$ in Orlicz space. For this purpose, let $\psi:[0,\infty) \to [0,\infty)$ be a generalized convex $\Phi$-function satisfying \eqref{compact:cond}.
Then, the set given by
\[L^\psi(\Omega)= \bigg\{u\in M(\Omega) \,:\, \into \psi(|\lambda u|) ~dx < \infty \ \text{for some} \ \lambda>0\bigg\}\] 
equipped with the norm
\[\|u\|_\psi := \inf\bigg\{\lambda>0 \,:\, \into \psi\l(\abs{\frac{u}{\lambda}}\r) ~dx \leq 1 \bigg\}\]
is a Banach space (see \cite[Theorem 3.3.7]{Harjulehto-Hasto-2019}).
\begin{theorem}\label{thm:compact}
    The embedding $\h(\Omega)\hookrightarrow L^\psi(\Omega)$ is compact.
\end{theorem}
\begin{proof}
Consider the inclusion map
\[
\mathcal{I}: \h(\Omega) \to L^\psi(\Omega) \quad \text{defined as} \quad \mathcal{I}(u)=u.
\]
Note that by Theorem \ref{thm:emb-cts} and \cite[Theorem 3.2.6]{Harjulehto-Hasto-2019}, we have $\h(\Omega) \hookrightarrow  L^\ph(\Omega) \hookrightarrow L^\psi(\Omega).$ Therefore, the inclusion map is well defined and continuous. Let $\{f_n\}_{n \in \mathbb{N}}$ be a bounded sequence in $\h(\Omega).$ By the reflexivity of $\h(\Omega)$ and \cite[Theorem 2.1]{Correa-DePablo-2018}, there exists a subsequence $\{f_n\}_{n \in \N}$ (denoted with same notations) such that $f_n \rightharpoonup f$ in $\h(\Omega)$ and $f_n \to f$ in $L^2(\Omega).$  In order to show the compactness of the above inclusion map, in view of Lemma \ref{norm_modular}, it is enough to show that for every $\eps >0$ there exists a $n_0 \in \mathbb{N}$ such that 
\[
\into \psi(|f_n-f|) ~dx \leq \eps \quad \text{for all} \ n \geq n_0.\]
Note that by \eqref{compact:cond}, for every $\eps_1>0$, there exist $\delta_1, \delta_2>0$ such that \begin{equation}\label{limits:est}
    \psi(t)\leq C(\epsilon_1) t  \quad \text{for} \ 0 < t < \delta_1 \quad \text{and} \quad \psi(t)\leq \eps_1 t^2\ln(e+t) \quad \text{for} \ t>\delta_2.
\end{equation}
Now, by using the fact that $\psi$ is an increasing function and \eqref{limits:est}, we obtain
\[
\begin{split}
     \into \psi(|f_n-f|) ~dx & \leq C(\epsilon_1) \int_{\{|f_n-f|<\delta_1\}} |f_n-f| ~dx\\
   & \qquad + \frac{\psi(\delta_2)}{\delta_1}\int_{\{\delta_1\leq |f_n-f| \leq \delta_2\}} |f_n-f|  ~dx\\ &\qquad + \eps_1 \int_{\{|f_n-f|>\delta_2\}} |f_n-f|^2 \ln (e+|f_n-f|) ~dx \\
   &\leq \l(C(\epsilon_1)  + \frac{\psi(\delta_2)}{\delta_1} \r)\|f_n-f\|_{L^1(\Omega)} \\ &\qquad  + \eps_1 \into |f_n-f|^2 \ln (e+|f_n-f|) ~dx.
\end{split}
\]
Since $f_n \to f$ in $L^1(\Omega)$, for every $\eps >0$, there exists a $n_0 \in \mathbb{N}$ such that
\begin{equation}\label{limits:est-1}
\|f_n-f\|_{L^1(\Omega)} \leq \frac{\eps}{2 \l(C(\epsilon_1)  + \frac{\psi(\delta_2)}{\delta_1} \r)} \quad \text{for all} \ n \geq n_0 
\end{equation}
Moreover, by using Theorem \ref{thm:emb-cts} and the fact that $\{f_n\}_{n \in \mathbb{N}}$ is a bounded sequence in $\h(\Omega)$, we get
\begin{equation}\label{limits:est-2}
\into |f_n-f|^2 \ln (e+|f_n-f|) ~dx \leq C \quad \text{for some} \ C>0.
\end{equation}
Finally, by using \eqref{limits:est-1}-\eqref{limits:est-2} and taking $\eps_1= \frac{\eps}{2C}$ in \eqref{limits:est}, we obtain the claim. Hence, the compact embedding of $\h(\Omega)$ in $L^\psi(\Omega)$ follows.
\end{proof}

\subsection{Sharpness of embeddings}
\begin{theorem}\label{them:non-compact}
Let $\Omega\subseteq \R^N$ be a bounded, convex domain. Then, the embedding $\h(\Omega) \hookrightarrow L^\ph(\Omega)$ is not compact.
\end{theorem}
\begin{proof}
Without loss of generality, we can assume that $\Omega$ contains 0. Let $k\in\N$ and $u\in \h(\Omega)$ such that $\into |u|^2 \ln (e+|u|) ~dx =1.$ For $k \in \mathbb{N} \setminus \{1\}$, define
\begin{equation}\label{def:seq-of:functions}
    u_k(x)= \begin{cases} 
      {\l(\frac{k^N}{\ln k}\r)}^\frac{1}{2}u(kx) & x\in \frac{\Omega}{k}:=\{\frac{x}{k} \,:\, x \in \Omega\} \\
      0 & \text{otherwise}
   \end{cases}
\end{equation}
such that $u_k \to 0$ a.e. in $\Omega.$ \vspace{0.1cm}\\
\textbf{Claim 1:} The sequence $\{u_k\}$ is bounded in $\h(\Omega).$ \vspace{0.1cm}\\
Note that
\[
\|u_k\|^2 = \frac{c_N k^N}{2 \ln k} \inta \frac{(u(kx)-u(ky))^2}{|x-y|^N} ~dx ~dy.
\]
By change of variables $s=kx$ and $t=ky$ and the symmetry of the integrand, we obtain
\[
\begin{split}
    \|u_k\|^2 
    & = \frac{c_N}{2 \ln k} \iint_{s,t \in \R^N,\ |s-t|\leq k} \frac{(u(s)-u(t))^2}{|s-t|^N} ~ds ~dt\\
    & = \frac{c_N}{2 \ln k} \bigg(\iint_{s,t \in \R^N, \ |s-t|\leq 1} \frac{(u(s)-u(t))^2}{|s-t|^N} ~ds ~dt\\ & \qquad 
 \qquad \qquad + \iint_{s,t \in \R^N, \ 1 < |s-t|\leq k} \frac{(u(s)-u(t))^2}{|s-t|^N} ~ds ~dt\bigg)\\
 & = \frac{1}{\ln k} \|u\|^2 +  \frac{c_N}{2 \ln k} \iint_{s,t \in \R^N, \ 1 < |s-t|\leq k} \frac{(u(s)-u(t))^2}{|s-t|^N} ~ds ~dt\\
 & \leq \frac{1}{\ln k} \|u\|^2 + \frac{c_N}{\ln k}\iint_{s,t \in \R^N,\ 1 < |s-t|\leq k} \frac{(u(s))^2+(u(t))^2}{|s-t|^N} ~ds ~dt \\
 & \leq \frac{1}{\ln k} \|u\|^2 +\frac{2 c_N}{\ln k} \int_{s\in \R^N}(u(s))^2 \l(\int_{t \in  \overline{B_k(s)} \setminus \overline{B_1(s)} } \frac{dt}{|s-t|^N}\r) ~ds\\
 & = \frac{1}{\ln k} \|u\|^2 +\frac{2 c_N}{\ln k} \int_{s\in \R^N}(u(s))^2 \l(\int_{y \in  \overline{B_k(0)} \setminus \overline{B_1(0)}} \frac{dy}{|y|^N}\r) ~ds\\
 & = \frac{1}{\ln k} \|u\|^2 +\frac{2N \omega_N c_N}{\ln k} \|u\|_{L^2(\Omega)}^2 \l(\int_1^k \frac{dr}{r}\r) \\
 &\leq  \frac{1}{\ln 2} \|u\|^2 +2N \omega_N c_N\|u\|_{L^2(\Omega)}^2 < \infty,
\end{split}
\]
where $\omega_N$ denotes the volume of the unit ball in $\mathbb{R}^N.$\\
\textbf{Claim 2:} There exists a $C>0$ and $k_0 \in \mathbb{N} \setminus \{1\}$ such that $\|u_k\|_\ph \geq C$ for all $k \geq k_0.$ \vspace{0.1cm}\\
For $\beta \in (0,1),$ let $\Omega_\beta:= \{x\in \Omega \ | \  |u(x)| \geq \beta \}.$ For every $u \in \h(\Omega) \setminus \{0\}$, we can choose a $\beta = \beta(u) >0$ small enough such that $\Omega_\beta$ is non-empty. Then, we have
\[
\begin{split}
    \into |u_k|^2 \ln (e+|u_k|) ~dx &= \frac{k^N}{\ln k}\into |u(kx)|^2 \ln \l(e+{\l(\frac{k^N}{\ln k}\r)}^\frac{1}{2}|u(kx)|\r)  ~dx \\
    & = \frac{1}{\ln k} \into |u(y)|^2  \ln \l(e+{\l(\frac{k^N}{\ln k}\r)}^\frac{1}{2}|u(y)|\r)  ~dy\\
    & \geq \frac{1}{\ln k} \int_{\Omega_\beta} |u(y)|^2  \ln \l(e+{\l(\frac{k^N}{\ln k}\r)}^\frac{1}{2}|u(y)|\r)  ~dy\\
    & \geq \frac{1}{\ln k} \int_{\Omega_\beta} \beta^2  \ln \l(e+{\l(\frac{k^N}{\ln k}\r)}^\frac{1}{2}\beta\r)  ~dy\\
    & \geq \frac{\beta^2}{2\ln k} \ln {\l(\frac{k^N \beta ^2}{\ln k}\r)} \ |\Omega_\beta|\\
    & = \frac{\beta^2 |\Omega_\beta|}{2} \bigg(N+\frac{\ln \beta^2}{\ln k}-\frac{\ln(\ln k)}{\ln k}\bigg).
\end{split}
\]
Note that $\frac{\ln(\ln k)}{\ln k} \to 0$ as $k \to \infty.$ Thus, there exist constants $C_0 >0$ and $k_0 \in \mathbb{N} \setminus \{1\}$ depending upon $\beta, N$ such that 
\[
N+\frac{\ln \beta^2}{\ln k}-\frac{\ln(\ln k)}{\ln k} \geq C_0 >0 \quad \text{for all} \ k \geq k_0.
\]
This further implies,
\begin{equation}\label{lower:est:norm}
    \into |u_k|^2 \ln (e+|u_k|) ~dx\geq C(N, \beta):= C_0 \frac{\beta^2 |\Omega_\beta|}{2} \quad \text{for all} \ k \geq k_0.
\end{equation}
Now, we proceed by contradiction. Suppose that the embedding  $\h(\Omega) \hookrightarrow L^\ph(\Omega)$ is compact. By \textbf{Claim 1}, Lemma \ref{norm_modular} and  \eqref{lower:est:norm}, there exists a $\tilde{u} \in \mathbb{H}(\Omega) \setminus \{0\}$ such that $u_k \rightharpoonup \tilde{u}$ in $\h(\Omega)$, $u_k \to \tilde{u}$ in $L^\ph(\Omega) \cap L^2(\Omega)$ and $u_k \to \tilde{u} \neq 0$ a.e. in $\Omega.$ But this contradicts the fact that $u_k \to 0$ a.e. in $\Omega$ and $\tilde{u} = 0$ a.e. in $\Omega.$
Hence, the theorem is proved.
\end{proof}
Next we are going to show that the space $\h(\Omega)$ is not continuously embedded in a smaller space $L^\gamma(\Omega)$ than $L^\ph(\Omega)$. In particular, we show that the inclusion map
\[
\mathcal{I}: \h(\Omega) \to L^\gamma(\Omega) \quad \text{defined as} \quad \mathcal{I}(u)=u
\]
does not map bounded sets to bounded sets. Let $\gamma \colon [0,\infty) \to [0,\infty)$ be a generalized convex $\Phi$-function satisfying \eqref{embd:cond:smaller}. Then, $L^\gamma(\Omega)$ is a Banach space associated with $\gamma.$
   It is straightforward to see  via \cite[Theorem 3.2.6]{Harjulehto-Hasto-2019} that $L^\gamma(\Omega) \hookrightarrow L^\ph(\Omega).$
\begin{theorem}
Let $\Omega\subseteq \R^N$ be a bounded and convex domain.  Then, $\h(\Omega)$ is not continuously embedded in  $L^\gamma(\Omega).$
\end{theorem}
\begin{proof}
Without loss of generality, we can assume that $\Omega$ contains 0. Let $u \in \h(\Omega)$ such that $\Omega_\beta:=\l\{ x \in \Omega: |u(x)| \geq \beta \r\} \neq \emptyset.$ Define 
\[
A_u= \{u_k : k \in \mathbb{N} \setminus \{1\}\},
\]
where $\{u_k\}$ is defined in \eqref{def:seq-of:functions}. By \textbf{Claim 1} in Theorem \ref{them:non-compact}, the set $A_u$ is bounded in $\h(\Omega).$ Next, we show that the set $A_u$ is not bounded in $L^\gamma(\Omega).$
On the contrary, we assume that there exists a constant $C>0$ such that
   \begin{equation}
   \label{bdd_gamma}
   \into \gamma (|u_k|) ~dx \leq C \quad \text{for all} \  u_k \in A_u.
   \end{equation}  
The property \eqref{embd:cond:smaller} implies that for every $M>0$, there exists a $\ell_0 \gg e$ depending upon $M$ and $\beta$ such that
\begin{equation}\label{lower:est:gamma1}
    \gamma(\ell \beta)> M \ph(\ell \beta) \quad \text{for all} \ |\ell|>\ell_0.
\end{equation} 
Choose $k_0 \in \mathbb{N}$ large enough such that 
\[
\l(k^N/\ln k\r)^\frac{1}{2} \frac{|u(x)|}{\beta} \geq \l(k^N/\ln k\r)^\frac{1}{2} > \ell_0 \quad \text{for all} \ k \geq k_0 \ \text{and} \ x \in \Omega_\beta.
\]
Applying change of variables and using \eqref{lower:est:gamma1} for $k \geq k_0$, we obtain
\begin{equation}\label{lower:est:gamma2}
    \begin{split}
   \into \gamma(|u_k|)~dx & = \frac{1}{k^N} \into \gamma\l({\l(\frac{k^N}{\ln k}\r)}^\frac{1}{2} |u(x)|\r)~dx \\
   & \geq \frac{1}{k^N} \int_{\Omega_\beta} \gamma\l({\l(\frac{k^N}{\ln k}\r)}^\frac{1}{2} |u(x)|\r)~dx \\
   & > \frac{M}{\ln k} \int_{\Omega_\beta} |u|^2 \ln \l(e+{\l(\frac{k^N}{\ln k}\r)}^\frac{1}{2} \beta \frac{|u(x)|}{\beta}\r) ~dx\\
   &  > \frac{M \ln \l[\beta{\l(\frac{k^N}{\ln k}\r)}^{\frac{1}{2}}\r]}{\ln k}\int_{\Omega_\beta} |u(x)|^2  ~dx + \frac{M}{\ln k}\int_{\Omega_\beta} |u(x)|^2 \ln \l(\frac{|u(x)|}{\beta} \r) ~dx  \\
       &> M \frac{|\Omega_\beta| \beta^2}{2} \l(N - \frac{\ln(\ln k)}{\ln k}+\frac{2\ln \beta}{\ln k}\r).
   \end{split}
\end{equation}
Note that $\frac{\ln(\ln k)}{\ln k} - \frac{ 2 \ln \beta}{\ln k} \to 0$ as $k \to \infty,$ {\it i.e.,} there exists a $k_1 \in \mathbb{N}$ large enough such that 
\[
\frac{\ln(\ln k)}{\ln k} - \frac{ 2 \ln \beta}{\ln k} \leq \frac{N}{2} \quad \text{for all} \ k \geq k_1.
\]
Finally, by choosing $k \geq \max\{k_0, k_1\}$ and $M = \frac{8C}{\beta^2 |\Omega_\beta|N}$ in \eqref{lower:est:gamma1} and \eqref{lower:est:gamma2}, we get
\[\into \gamma(|u_k|)~dx > M \frac{|\Omega_\beta| \beta^2 N}{2} > 2C\]
which is a contradiction. Hence, $A_u$ is not bounded in $L^\gamma(\Omega)$ and $\h(\Omega)$ is not continuously embedded in $L^\gamma(\Omega).$
\end{proof}

\section{Brezis-Nirenberg type problem involving the logarithmic Laplacian}\label{brezis-Nirenberg}
We begin this section by recalling the fractional Sobolev inequality (see \cite[Theorem 1.1]{Cotsiolis-Tavoularis-2004}) and the sharp logarithmic Sobolev inequality (see \cite[Proposition 3.8]{Santamaria-Saldana-2022}).  
\begin{theorem}
    \label{Frac_Sobo_ineq}
{(Fractional Sobolev inequality)} Let $N\geq 1, s\in (0,\frac{N}{2})$, and  $2^*_s:= \frac{2N}{N-2s}$. Then,
\[
\begin{split}
\|u\|^2_{L^{2^*_s}(\Omega)} \leq \kappa_{N,s}\|u\|^2_s \quad \text{ for all}\  u\in H^s(\R^N),
\end{split}
\] 
where $\kappa_{N,s}$ is defined as
\[ 
\begin{split}
\kappa_{N,s} = 2^{-2s} \pi^{-s}\frac{\Gamma(\frac{N-2s}{2})}{\Gamma(\frac{N+2s}{2})}.{\l(\frac{\Gamma(N)}{\Gamma(\frac{N}{2})}\r)}^{\frac{2s}{N}}.
\end{split}
\]
\end{theorem}

\begin{proposition}
    \label{Log-Sobolev-ineq}
    (Sharp logarithmic Sobolev inequality)
For every $u\in \h(\Omega)$,
\[
\begin{split}
\frac{4}{N} \into \ln (|u|)u^2 ~dx \leq \mathcal{E}_L(u,u)+\frac{4}{N}\ln (\|u\|_{L^2(\Omega)})\|u\|^2_{L^2(\Omega)} + a_N \|u\|^2_{L^2(\Omega)},
\end{split}
\]
where
\[a_N:= \frac{2}{N} \ln \l(\frac{\Gamma(N)}{\Gamma(\frac{N}{2})}\r) - \ln (4\pi) - 2\psi \l(\frac{N}{2}\r) \\
\text{and} \ \psi \ \text{is the digamma function.}
\]
Moreover, 
\[
\|u\|_{L^2(\Omega)}^2 \leq S_{ln} \|u\|^2, \quad S_{ln}:= \inf_{u \in \mathbb{H}(\Omega), \|u\|_{L^2(\Omega)}=1} \mathcal{E}(u,u).
\]
\end{proposition}
\subsection{Existence result for \eqref{prob:logarithmic-laplace} with $\sigma \in (0, \frac{4}{N})$} Throughout this subsection, we assume $\sigma \in (0, \frac{4}{N}).$ The assumptions \eqref{assump f1}-\eqref{assump f3} imply that for a given $\eps>0$ there exist constants $C>0$ and $C_0>0$, depending upon $\eps$, $\|g\|_{L^\infty(\Omega)}$ such that
\begin{equation}
\label{assumption2}
0 \leq |f(x,t)| \leq C_0 |t|+\eps |t| {|\ln|t||} \quad \text{for all} \ t \in \mathbb{R} \ \text{and} \ x \in \Omega
\end{equation}
and
\begin{align}
\label{assump F1}
    \begin{split}
   0 \leq |F(x,t)| \leq  C t^2 + \eps t^2 |\ln|t||  \quad \text{for all} \ t \in \mathbb{R} \ \text{and} \ x \in \Omega.
    \end{split}
\end{align}
By following the same arguments as in \cite[Lemma 3.9]{Santamaria-Saldana-2022} and using \eqref{assump f1} and \eqref{assump f2}, we have the following lemma.
\begin{lemma}\label{lem:C^1regu}
    Let $\mathbb{E}$ be given by \eqref{def:enerfunct:plambda} and $f$ satisfies \eqref{assump f1} and \eqref{assump f2}. 
    Then, the energy functional $\mathbb{E}$ is of class $C^1$ in $\mathbb{H}(\Omega)$ and 
    \[
    (\mathbb{E}'(u), v)_{\mathbb{H}} = \mathcal{E}_L(u,v) - \into (f(x,u) +\sigma u \ln\abs{u}) v  ~dx \quad  \text{for  every} \ v\in \h(\Omega).
    \]
    In particular, $\mathbb{E}'(u)\in \mathcal{L}(\h(\Omega),\R) $ for any $u \in \h(\Omega)$ and  $\mathbb{E}':\h(\Omega) \to \h(\Omega)$ is continuous. Here, $\mathcal{L}(\h(\Omega),\R)$ represents the set of all bounded linear functionals on $\h(\Omega)$.
\end{lemma}

We begin with proving that the Nehari manifold $N_{0,\sigma}$ is non-empty and is away from the origin in $\h(\Omega)$, {\it i.e.,} $N_{0,\sigma} \cap B_c(0) = \emptyset$ for some $c>0$, where $B_c(0) = \{u \in \h(\Omega): \|u\| >c\}.$ For this purpose, we define the fibering map $n_u: (0, \infty) \to \mathbb{R}$ such that 
\[
n_u(r) := \mathbb{E}(ru) \quad \text{for} \ u\in \h(\Omega) \setminus \{0\}.
\]
It is easy to verify that 
\[
N_{0, \sigma} = \{r u \in \h(\Omega) \setminus \{0\}: n'_u(r) =0 \}.
\]
\begin{lemma}
    \label{Nehari-struc} 
 Assume that $f$ satisfies \eqref{assump f1}, \eqref{assump f2} and \eqref{assump f3}. Then, for every $u\in \h(\Omega) \setminus \{0\}$,
   there exists a point $r_{0,u}$ such that $n_u'(r_{0,u}) =0$ and $n_u'(r) >0$ for $r \in (0, r_{0,u}).$ Moreover, $r_{0,u} u\in N_{0, \sigma}.$ 
\end{lemma}

\begin{proof}
Let $u \in \mathbb{H}(\Omega) \setminus \{0\}$ and $r>0.$ Note that
    \[
    \begin{split}
      n_u(r) &= \frac{r^2}{2}\mathcal{E}_L(u,u)-\into F(x,ru) ~dx+ \frac{\sigma r^2}{4} \into u^2 ~dx -\frac{\sigma r^2 \ln r}{2} \into u^2 ~dx\\
      & \qquad \qquad \qquad \qquad-\frac{\sigma r^2}{2} \into u^2 \ln|u| ~dx.
    \end{split}
    \]
Using \eqref{assump F1} with $\eps < \frac{\sigma}{2}$ in the definition of $n_u$, we obtain, for $r \in (0,1)$
{ \[
\begin{split}
    n_u(r)& \geq r^2\bigg(\frac{1}{2} \mathcal{E}_L(u,u)-\bigg[\l(\frac{\sigma}{2} \ln r + \eps |\ln r| \r)-\frac{\sigma}{4} + C\bigg]\|u\|^2_{L^2(\Omega)} \\
    & \qquad \qquad \qquad \qquad \qquad - \frac{\sigma}{2} \into u^2 \ln|u| ~dx -  \eps \into u^2 |\ln|u|| ~dx\bigg)\\
    & \geq r^2\bigg(\frac{1}{2} \mathcal{E}_L(u,u)-\bigg[\l(\frac{\sigma}{2} - \eps \r)\ln r-\frac{\sigma}{4} + C\bigg]\|u\|^2_{L^2(\Omega)} \\
    & \qquad \qquad \qquad \qquad \qquad - \frac{\sigma}{2} \into u^2 \ln|u| ~dx -  \eps \into u^2 |\ln|u|| ~dx\bigg)\\
\end{split}
\]
and for any $r \geq 0$
\[
\begin{split}
    n_u(r)& \leq r^2\bigg(\frac{1}{2} \mathcal{E}_L(u,u)-\bigg[\l(\frac{\sigma}{2} \ln r - \eps |\ln r|\r) -\frac{\sigma}{4}-C\bigg] \|u\|^2_{L^2(\Omega)} \\
    & \qquad \qquad \qquad \qquad - \frac{\sigma}{2} \into u^2 \ln|u| ~dx +  \eps \into u^2 |\ln|u|| ~dx\bigg).
\end{split}
\]}
This implies that there exists $r_{0,u}>0$ small enough such that
\[ \lim\limits_{r \to 0} n_u(r)=0, \quad \lim\limits_{r \to \infty} n_u(r)=-\infty \quad \text{and} \quad n_u(r)>0 \ \text{for} \ r \in (0, r_{0,u}).
\] Thus, there exists a (first) critical point $r_{0,u}$ of $n_u(\cdot)$ such that $n_u'(r) >0$ for $r \in (0, r_{0,u})$. Moreover, $r_{0,u} u \in N_{0,\sigma}.$
\end{proof}
\begin{lemma}
    \label{lower-bds-norms}
Assume $f$ satisfies \eqref{assump f1}, \eqref{assump f2} and \eqref{assump f3}.  Then, there exist constants $c_1>0$ and $c_2>0$ such that
\[
\|u\|_{L^2(\Omega)}\geq c_1 \quad \text{and} \quad \|u\|>c_2 \quad \text{for all} \ u \in N_{0, \sigma}.
\]
\end{lemma}
\begin{proof}
{ Note that,
\begin{equation}\label{nonl-reminder-term-est}
    \begin{split}
    \into u^2 \ln |u| {\bf 1}_{|u| \leq 1} ~dx &= \ln \|u\|_{L^2(\Omega)} \into u^2 {\bf 1}_{|u| \leq 1} ~dx \\
    & \qquad \qquad + \into u^2 {\bf 1}_{|u| \leq 1} \ln \left(\frac{|u|}{\|u\|_{L^2(\Omega)}}\right) ~dx\\
    & \geq \|u {\bf 1}_{|u| \leq 1}\|_{L^2(\Omega)}^2 \ln \|u\|_{L^2(\Omega)} - \frac{|\Omega|}{2e} \|u\|_{L^2(\Omega)}^2
\end{split}
\end{equation}
Let $u \in N_{0, \sigma}$ such that $\|u\|_{L^2(\Omega)} < c_1:=\min\{c_\ast,1\}$ where
\[c_\ast:= \exp{\l(\frac{1}{(\sigma-\eps)}\l[\l(1-\frac{(\sigma+ \eps) N}{4}\r)\lambda_L^1-C-\frac{ \eps|\Omega|}{e}-\frac{(\sigma +\eps) N a_N}{4}\r]\r)}\]
and $C, \eps$ are defined in \eqref{assump F1}.} Now, by using \eqref{assump F1} with $\eps < \frac{4}{N}-\sigma$, \eqref{nonl-reminder-term-est} and Proposition \ref{Log-Sobolev-ineq}, we have
    \[
    \begin{split}
       & (\mathbb{E'}(u),u)_{\h} \geq \mathcal{E}_L(u,u)- C \|u\|^2_{L^2(\Omega)} - { (\sigma +\eps)\into u^2\ln|u| {\bf 1}_{|u| \geq 1} ~dx} \\
       & \quad \quad \quad \quad \quad \quad \quad \quad{ - (\sigma-\eps) \into u^2 \ln |u| {\bf 1}_{|u| \leq 1} ~dx}\\
       & \quad \geq  \mathcal{E}_L(u,u)- C \|u\|^2_{L^2(\Omega)} -(\sigma + \eps)\into u^2\ln|u| ~dx + { 2 \eps \into u^2\ln|u| {\bf 1}_{|u| \leq 1} ~dx}\\
       & \quad \geq \mathcal{E}_L(u,u)- \frac{(\sigma + \eps) N}{4} \l[\mathcal{E}_L(u,u)+\frac{2}{N}(\ln \|u\|^2_{L^2(\Omega)}) \|u\|^2_{L^2(\Omega)}+a_N \|u\|^2_{L^2(\Omega)}\r] \\
        & \qquad \quad - C \|u\|^2_{L^2(\Omega)} { + 2 \eps \|u {\bf 1}_{|u| \leq 1} \|_{L^2(\Omega)}^2 \ln \|u\|_{L^2(\Omega)} - \frac{ \eps|\Omega|}{e} \|u\|_{L^2(\Omega)}^2 }\\
        &\quad = { \l[\l(1- \frac{(\sigma + \eps) N}{4}\r)\lambda_L^1\r] \|u\|^2_{L^2(\Omega)}} \\
        & \qquad \qquad - { \l(C+ \frac{ \eps|\Omega|}{e} + \frac{(\sigma-\eps)}{2}\ln (\|u\|_{L^2(\Omega)}^2)+\frac{(\sigma + \eps) N a_N}{4}\r) \|u\|^2_{L^2(\Omega)}} >0,
    \end{split}
    \]
where $\lambda_L^1$ represents the first Dirichlet eigenvalue of $L_\Delta$ in $\Omega.$  But if $u \in N_{0,\sigma}$ then $ (\mathbb{E'}(u),u)_{\h} =0$. Hence, $\|u\|_{L^2(\Omega)} \geq c_1.$ Since $\h(\Omega)\hookrightarrow L^2(\Omega)$ is compact, we have
$$c_1^2 \leq \|u\|_{L^2(\Omega)}^2 \leq S_{ln} \|u\|^2 \quad  \text{for all} \  u\in \h(\Omega).$$ Therefore, $\|u\| > c_2:=c_1S_{ln}^\frac{-1}{2}$. Hence the claim follows.
\end{proof}
\begin{proposition}
    \label{weak-conve-est}
    Let $f$ satisfies \eqref{assump f1}, \eqref{assump f2}, \eqref{assump f3}, \eqref{assump f-sigma} and $(u_n)_{n\in \mathbb N} \subset N_{0,\sigma}$ be a sequence such that $\sup_{n\in \mathbb N} \mathbb{E}(u_n)\leq C'$ for some constant $C'>0$. Then, $(u_n)_{n\in \mathbb N}$ is bounded in $\h(\Omega)$. Moreover, passing to a subsequence, there exists a $ u_0 \in \h(\Omega)\setminus\{0\}$ such that $u_n \rightharpoonup u_0$ weakly in $\h(\Omega)$ and $u_n \to u_0$ strongly in $L^2(\Omega),$ as $n \to \infty$. Moreover, $\inf_{N_{0,\sigma}} \mathbb{E}$ exists.
    \end{proposition}
    \begin{proof}
        Since $u_n \in N_{0,\sigma}$ and \eqref{assump f-sigma} implies  \eqref{modi-AR-cond} as in Remark \ref{remark-AR-cond}, we get
        \begin{equation}\label{lower:est-Nehari}
            \begin{split}
        \mathbb{E}(u_n)&= \frac{1}{2}\into f(x,u_n)u_n-\into F(x,u_n) ~dx+\frac{\sigma}{4}\into u_n^2 ~dx\\ & \geq \frac{(\sigma-\delta)}{4}\|u_n\|_{L^2(\Omega)}^2 >0.
        \end{split}
        \end{equation}
Therefore, by using $\mathbb{E}(u_n) \leq C'$ and the above estimate, we obtain       
 \begin{equation}\label{bound-on-L^2-norm}
            \sup_{n \in \mathbb N} \|u_n\|^2_{L^2(\Omega)} \leq  C_1:=\frac{4C'}{(\sigma-\delta)}.
            \end{equation}
By \eqref{assump F1}, \eqref{nonl-reminder-term-est}, Proposition \ref{Log-Sobolev-ineq}, and taking $\eps < \frac{1}{2}\l(\frac{4}{N}-\sigma\r)$, we have
        \[
    \begin{split}
         \mathbb{E}(u_n) &= \frac{1}{2}\mathcal{E}_L(u_n,u_n)-\frac{\sigma}{4}\into u_n^2(\ln |u_n|^2-1) ~dx- \into F(x,u_n) ~dx \\
         & \geq \frac{1}{2}\mathcal{E}_L(u_n,u_n) - { \l(C-\frac{\sigma}{4}\r) \|u_n\|^2_{L^2(\Omega)} - \l(\frac{\sigma}{2} +\eps\r)\into u_n^2 \ln |u_n| ~dx }\\
         & \qquad \qquad {+ 2 \eps \into |u_n|^2 \ln |u_n| {\bf 1}_{|u_n| \leq 1}~dx}\\
         & \geq \frac{1}{2}\mathcal{E}_L(u_n,u_n) - { \l(C-\frac{\sigma}{4}\r) \|u_n\|^2_{L^2(\Omega)} - \l(\frac{\sigma}{2} +\eps\r)\into u_n^2 \ln |u_n| ~dx }\\
         & \qquad \qquad { + 2 \eps \|u_n {\bf 1}_{|u_n| \leq 1}\|_{L^2(\Omega)}^2 \ln \|u_n\|_{L^2(\Omega)} - \frac{ \eps|\Omega|}{e} \|u\|_{L^2(\Omega)}^2}\\
         &\geq \frac{\mathcal{E}_L(u_n,u_n)}{2} \l(1- \frac{N (\sigma + 2\eps)}{4}\r) \\
         & \qquad \qquad -{ \l(\frac{(\sigma+2\eps)}{4} \|u_n\|^2_{L^2(\Omega)} + \eps \|u_n {\bf 1}_{|u_n| \leq 1}\|_{L^2(\Omega)}^2\r)\ln \|u_n\|_{L^2(\Omega)}^2} \\
         & \qquad \qquad  -{ \l(C+ \frac{\eps |\Omega|}{e}-\frac{\sigma}{4}+\frac{Na_N  (\sigma + 2\eps)}{8}\r)} \|u_n\|^2_{L^2(\Omega)}.
     \end{split}
     \]
Finally, by using $\sigma \in (0, \frac{4}{N})$, \eqref{bound-on-L^2-norm} and \cite[Lemma 3.4]{Santamaria-Saldana-2022} in the definition of $\mathcal{E}_L(\cdot,\cdot)$, we get
     \[\mathcal{E}(u_n,u_n)= \|u_n\|^2 < C_2,\]
     for a constant $C_2>0$
     which in turn implies that  $(u_n)_{n\in \mathbb N}$ is bounded in $\h(\Omega)$. 
     Thus, we can get a $u_0 \in \h(\Omega)$ such that $u_n \rightharpoonup u_0$ in $\h(\Omega)$. 
    Now, by the compact inclusion of $\h(\Omega)$ in $L^2(\Omega)$, $u_n \to u_0$ in $L^2(\Omega)$. The non-triviality of $u_0$ follows from Lemma \ref{lower-bds-norms}. Moreover, by \eqref{lower:est-Nehari}, $\inf_{N_{0,\sigma}} \mathbb{E}$ exists.
     \end{proof}
\begin{lemma}
\label{limits_interchange}
Let $f$ satisfies \eqref{assump f1}-\eqref{assump f4} and $(u_n)_{n\in \N}\subseteq \h(\Omega)$ such that $u_n \rightharpoonup u_0$ for some $u_0 \in \h(\Omega)$. Then, the following assertions hold true:
    \begin{enumerate}
        \item [\textnormal{(i)}] If $(\zeta_n)_{n \in \N} \subseteq C_c^\infty(\Omega),$ such that $\zeta_n \to \zeta_0$ in $\h(\Omega)$ as $n \to \infty,$ then 
        \[\lim\limits_{n \to \infty} \into f(x,\zeta_0)\zeta_n ~dx = \into f(x,\zeta_0) \zeta_0 ~dx. \]
        
        \item  [\textnormal{(ii)}] $$\lim\limits_{n \to \infty} \into f(x,u_n)\ph ~dx = \into f(x,u_0)\ph ~dx \quad \text{for all} \ \ph \in C_c^\infty(\Omega).$$
        
         \item [\textnormal{(iii)}] $$\lim\limits_{n \to \infty} \into f(x,u_n)u_n ~dx = \into f(x,u_0)u_0 ~dx$$ and $$\lim\limits_{n \to \infty} \into F(x,u_n) ~dx = \into F(x,u_0) ~dx.$$

    \end{enumerate}
\end{lemma}

\begin{proof}
By H\"older inequality in \cite[Lemma 3.2.11]{Harjulehto-Hasto-2019}, we have
\[
\into |f(x,\zeta_0)||\zeta_n-\zeta_0| ~dx \leq 2\|f(x,\zeta_0)\|_{L^{\ph^\ast}(\Omega)}\|\zeta_n-\zeta_0\|_{L^\ph(\Omega)}
\]
where $\ph^\ast$ denotes the H\"older conjugate of $\ph$.
Since $\zeta_n \to \zeta_0$ in $\h(\Omega),$ by Theorem \ref{thm:embd:results}, $\zeta_n \to \zeta_0$ in $L^\ph(\Omega).$ To prove our claim, it thus remains to show that $f(\cdot,\zeta_0) \in L^{\ph^\ast}(\Omega).$ For this, it is enough to show that $\into \ph^\ast(|f(x,\zeta_0)|) ~dx <\infty.$
{ Using \eqref{assump f1}-\eqref{assump f3}, we get}
\begin{equation}\label{upper:est-bound}
    \begin{split}
    |f(x,\zeta_0)| \leq C |\zeta_0|\ln(e+|\zeta_0|),
\end{split}
\end{equation}
where $C_\eps>0$. Since $\ph$ is increasing and fulfills (Dec)$_{2+\eps}$, by \cite[Proposition 2.4.9]{Harjulehto-Hasto-2019}, $\ph^\ast$ is an increasing function. Now, by using \cite[Theorem 2.4.8 and Corollary 2.4.11]{Harjulehto-Hasto-2019} and \eqref{upper:est-bound}, we obtain
\[
\begin{split}
    \ph^\ast(|f(x,\zeta_0)|) &\leq \ph^\ast(C_\eps |\zeta_0|\ln(e+|\zeta_0|)) \\ &\leq \tilde{C} \ph^\ast(|\zeta_0|\ln(e+|\zeta_0|)) = \tilde{C} \ph^\ast\l(\frac{\ph(|\zeta_0|)}{|\zeta_0|}\r)\leq \tilde{C} \ph(|\zeta_0|),
\end{split}
\]
for some constant $\tilde{C} >0.$ Thus,
\[ \into \ph^\ast(|f(x,\zeta_0)|) ~dx \leq  \tilde{C} \into \ph(|\zeta_0|) ~dx <\infty.\]
Hence, \[\lim\limits_{n \to \infty} \into f(x,\zeta_0)\zeta_n ~dx = \into f(x,\zeta_0)\zeta_0 ~dx.\] This completes the proof of (i). Next, we prove (ii). Let $u_n \rightharpoonup u_0$ in $\h(\Omega).$ Then, by \cite[Theorem 2.1]{Correa-DePablo-2018}, $u_n \to u_0$ in $L^2(\Omega)$ and a.e. in $\Omega$, and for $\eta >0.$ Since $u_n \to u_0$ in $L^2(\Omega)$, there exists a $\delta>0$ such that 
\begin{equation}\label{equi-inte-2}
\int_{A} |u_n|^2 ~dx \leq \frac{\eta}{3C_0} \quad \text{for all} \ A \subset \Omega, \ |A| < \delta
\end{equation}
where $C_0$ is defined in \eqref{assumption2}.
Now, using \eqref{assumption2} and the fact that $\ln |t| < |t|$ for $|t| \geq 1$, we get
    \[
    \begin{split}
       & |f(x,u_n)\ph| \leq \l( C_0 |u_n| { + \eps |u_n| |\ln|u_n|| \chi_{\{|u_n| < 1\}}} + \eps |u_n|\ln|u_n| \chi_{\{|u_n| \geq 1\}} \r)\|\ph\|_{L^\infty(\Omega)}\\
& \qquad \leq { C \l(|u_n|+ \eps + |u_n|^{2}\r)\|\ph\|_{L^\infty(\Omega)} }\in L^1(A).
    \end{split}
    \]
    This implies $(f(x,u_n) \phi)_{n \in \mathbb{N}}$ is uniformly equi-integrable. Moreover, since $f$ satisfies \eqref{assump f4} and $u_n \to u_0$ a.e. in $\Omega$, $f(x, u_n) \to f(x,u_0)$ a.e. in $\Omega$. Then, by Vitali's theorem, (ii) follows. Note that $u_n$ is bounded in $\h(\Omega)$ and $L^2(\Omega)$, therefore there exists a $C_2>0$ such that
\[
 \mathcal{E}_L(u_n,u_n)+\frac{4}{N}\ln (\|u_n\|_2)\|u_n\|^2_2 + a_N \|u_n\|^2_2 \leq C_2.
\]
Now, by using \eqref{assumption2} with $\eps = \min\left\{ \frac{4\eta}{3N C_2}, \frac{2e \eta}{3|\Omega|}\right\}$, \eqref{equi-inte-2} and Proposition \ref{Log-Sobolev-ineq}, we obtain
\begin{equation}\label{unifintegrability}
    \begin{split}
  \int_A f(x,u_n) & u_n ~dx \leq C_0 \int_A |u_n|^2 ~dx + \eps \int_A { |u_n|^2 |\ln|u_n||}~dx\\
  & \leq \frac{\eta}{3} { + \eps \int_{\Omega} |u_n|^2 \l(|\ln |u_n|| -\ln(|u_n|)\r) \chi_{\{|u_n| <1\}} ~dx} + \eps \into u_n^2 \ln |u_n| ~dx\\
  & \leq \frac{\eta}{3} + { \frac{\eps |\Omega|}{2e}}  + \frac{N \eps}{4} \l[ \mathcal{E}_L(u_n,u_n)+\frac{4}{N}\ln (\|u_n\|_2)\|u_n\|^2_2 + a_N \|u_n\|^2_2\r]\\
  & < \frac{\eta}{3} + \frac{\eta}{3} + \frac{\eta}{3} = \eta.
\end{split}
\end{equation}
Thus, $(f(\cdot, u_n) u_n)_{n \in \N}$ is uniformly equi-integrable in $L^1(\Omega)$. Using $u_n \to u_0$ a.e. in $\Omega$ and Vitali convergence theorem, we get
\[\lim\limits_{n \to \infty} \into f(x,u_n)u_n ~dx = \into f(x,u_0)u_0 ~dx.\]
Again, using \eqref{assump F1}, arguments in \eqref{unifintegrability} for $F$ and Vitali convergence theorem, we obtain
\[
\lim\limits_{n \to \infty} \into F(x,u_n) ~dx = \into F(x,u_0) ~dx.
\] This completes the proof of (iii). 
\end{proof}

\noindent \textit{Proof of Theorem \ref{Main-res-limitingprob}:} 
Let $\psi: \h(\Omega)\setminus\{0\} \to \R$ be given by
\[\psi(u):= \mathcal{E}_L(u,u)-\frac{\sigma}{2}\into u^2 \ln u^2 ~dx - \into f(x,u)u ~dx.\]
By a similar set of arguments as in Lemma \ref{lem:C^1regu} and \eqref{assump f4}, $\psi$ is a $C^1$ functional on $\h(\Omega)\setminus\{0\}.$
It is known that if $f: X \to Y$ is a $C^1$ map between manifolds $X$ and $Y$. The value $y \in Y$ is called a regular value if $\frac{df}{dx}: T_x(X) \to T_y(Y)$ is surjective at every point $x$ such that $f(x) = y.$ Here, $T_x(X), T_y(Y)$ are tangent spaces at $x,y$ respectively. In our case, for $v \in \h(\Omega)$
\[
(\psi'(u), v)_{\mathbb{H}} = 2 \mathcal{E}_L(u,v)-\sigma\into u v (\ln u^2+1) ~dx - \into f'(x,u)u v ~dx -\into f(x,u)v ~dx.
\]
Moreover, using \eqref{assump f-sigma}, we have
\begin{equation}\label{new-funct-lower-bounds}
    (\psi'(u), u)_{\mathbb{H}} = \l(\into f(x,u)u ~dx-\into f'(x,u)u^2 ~dx-\sigma \into u^2 ~dx\r)< 0\quad \text{if} \ u\in N_{0,\sigma}
\end{equation}
which in turn implies that $0$ is a regular value of $\psi$. Therefore, $N_{0,\sigma} = \psi^{-1} (0)$ is a $C^1$ manifold. It is easy to see that any minimizer $u$ of $\mathbb{E}$ restricted to $N_{0,\sigma}$ satisfies $\psi(u)=0.$ In view of this property, we can apply the critical point theory on $N_{0,\sigma}$ to get critical points of $\mathbb{E}.$ By Ekeland's variational principle \cite[Corollary 3.4]{Ekeland}, there is a minimizing sequence $\{u_n\}_{n\in \mathbb{N}} \subset N_{0,\sigma}$ and $\{\zeta_n\}_{n\in \mathbb{N}}\subset \R$ such that
\begin{equation}
    \label{variationalprinciple-1}
    0\leq \mathbb{E}(u_n)-\inf_{N_{0,\sigma}} \mathbb{E} \leq \frac{1}{n^2} \quad \text{and} \quad \|\mathbb{E}'(u_n)-\zeta_n \psi'(u_n)\|_{\mathcal{L}(\h(\Omega),\R)}\leq \frac{1}{n}.
\end{equation}
Moreover, $\mathbb{E}(u_n) < +\infty$ for all $n \in \mathbb{N}$ and as $n \to \infty$
 \[
 \begin{split}
 o(1)&= \frac{1}{\|u_n\|}\l((\mathbb{E}'(u_n),u_n)_{\h}-\zeta_n (\psi'(u_n),u_n)_{\h} \r)\\ 
&=\frac{1}{\|u_n\|}\l[-\zeta_n\l(\into f(x,u_n)u_n ~dx-\into f'(x,u_n)u_n^2 ~dx-\sigma \into u_n^2 ~dx\r)\r].
\end{split}
\]
By \eqref{new-funct-lower-bounds}, Lemma \ref{lower-bds-norms} and Proposition \ref{weak-conve-est}, we get that $\zeta_n \to 0$ as $n \to \infty$ and $C_2' \leq \|u_n\|^2<C_2$ , for all $n\in \mathbb N$ and for some constants $C_2, C_2'>0.$ Thus, there is a $u_0 \in \h(\Omega)\setminus \{0\}$ such that as $n\to \infty$, $u_n \rightharpoonup u_0 $ in $\h(\Omega)$ and $u_n \to u_0 $ in $L^2(\Omega).$ Again, by boundedness of the sequence $\{u_n\}$, \eqref{assump f4} and Proposition \ref{Log-Sobolev-ineq}, we have 
\[
\begin{split}
\abs{(\psi'(u_n), v)_{\mathbb{H}}} < C \quad \text{for} \ v \in \h(\Omega), \|v\| =1.
\end{split}
\]
Now, by \eqref{variationalprinciple-1}, we have
 \[\|\mathbb{E}'(u_n)\|_{\mathcal{L}(\h(\Omega),\R)}\leq \frac{1}{n}+\|\zeta_n \psi'(u_n)\|_{\mathcal{L}(\h(\Omega),\R)} \to 0 \ \text{as} \ n\to \infty.\]
By Lemma \ref{limits_interchange} (i) and \cite[Lemma 5.4]{Santamaria-Saldana-2022}, for $\ph \in C^\infty_c(\Omega)$, we obtain
 \begin{equation}\label{eq:weak-solu}
 \begin{split}
 0 &= \lim_{n \to \infty} (\mathbb{E}'(u_n), \ph)_{\mathbb{H}}\\
 &= \lim_{n\to\infty} \l(\mathcal{E}_L(u_n,\ph)-\frac{\sigma}{2}\into \ln(u_n^2)u_n\ph ~dx- \into f(x,u_n)\ph ~dx \r) \\
 &= \mathcal{E}_L(u_0,\ph)-\frac{\sigma}{2}\into \ln(u_0^2)u_0\ph ~dx -\into f(x,u_0)\ph ~dx .
 \end{split}
  \end{equation}
The density of $C_c^\infty(\Omega)$ in $\mathbb{H}(\Omega)$ implies that $u_0$ is a weak solution of \eqref{prob:logarithmic-laplace}. Assume, $(\ph_n)_{n\in \mathbb N} \subset C^\infty_c(\Omega)$ such that $\ph_n \to u_0$ in $\h(\Omega)$ as $n\to \infty$. Using Lemma \ref{limits_interchange} (i) and \eqref{eq:weak-solu}, we obtain
 \[
 \begin{split}
 0 &= \lim_{n\to \infty} \l(\mathcal{E}_L(u_0,\ph_n) - \frac{\sigma}{2}\into u_0 \ph_n \ln u_0^2  ~dx- \into f(x,u_0)\ph_n ~dx \r)\\
 &=\mathcal{E}_L(u_0,u_0) - \frac{\sigma}{2}\into u_0^2 \ln u_0^2 ~dx- \into f(x,u_0)u_0 ~dx.
 \end{split}
 \]
 Therefore, $u_0 \in N_{0,\sigma}$. Using $u_n , u_0 \in N_{0,\sigma}$, $u_n \to u_0$ in $L^2(\Omega)$, and Lemma \ref{limits_interchange} (ii), we get
 \begin{equation}\label{eq:ground-state-proof}
     \begin{split}
 \inf_{N_{0,\sigma}} \mathbb{E} &= \lim_{n \to \infty} \mathbb{E}(u_n)\\
 &= \lim_{n\to\infty} \l(\frac{1}{2}\into f(x,u_n)u_n ~dx- \into F(x,u_n) ~dx+\frac{\sigma}{4}\into u_n^2 ~dx\r) \\
 &=\l(\frac{1}{2}\into f(x,u_0)u_0 ~dx-\into F(x,u_0) ~dx+\frac{\sigma}{4}\into u_0^2 ~dx\r) = \mathbb{E}(u_0).
 \end{split}
 \end{equation}
Let $u_0$ be a least energy solution of \eqref{prob:logarithmic-laplace}. By \cite[Lemma 3.3]{Chen-Weth-2019}, we have $\abs{u_0} \in \h(\Omega)$,
 \begin{equation}
 \label{sign-change}
 \mathcal{E}_L(\abs{u_0},\abs{u_0})\leq \mathcal{E}_L(u_0,u_0),
 \end{equation}
and equality holds if and only if $u_0$ does not change sign. Moreover, if $F(x,t)\leq F(x,|t|)$ for all $t \in \mathbb{R}$, then we have
\[
\begin{split}
    \mathbb{E}(|u_0|)&=\frac{1}{2}\mathcal{E}_L(|u_0|,|u_0|)- \into F(x,|u_0|) ~dx-\frac{\sigma}{4}\into |u_0|^2(\ln(|u_0|^2)-1) ~dx\\
    &\leq\frac{1}{2}\mathcal{E}_L(u_0,u_0)- \into F(x,u_0) ~dx-\frac{\sigma}{4}\into u_0^2(\ln(u_0^2)-1) ~dx\\
    &=\mathbb{E}(u_0)=\inf_{N_{0,\sigma}} \mathbb{E}.
\end{split}
\]
Thus, $\mathbb{E}(|u_0|)=\mathbb{E}(u_0)$. Hence, \eqref{sign-change} holds with equality and we can conclude that $u_0$ has a constant sign in $\Omega.$ 
\qed
\begin{remark}
In the above result, if $f(x,t)= \omega(x) t$ for $\omega \in L^\infty(\Omega)$, it is interesting to note that even without any sign assumption on the weight $\omega$, all least energy solutions do not change the sign in the domain $\Omega.$ 
We also point out that even in the case of the Laplacian, the non-linearity of the type $\sigma |u|^{p-2}u$ and $\sigma u \ln u$ behave significantly differently, for a detailed description we refer the reader to \cite{Shuai-2023} and the references therein.
\end{remark}
\subsection{Uniform asymptotics estimates}
Here, we derive uniform estimates on the sequence of solutions $u_s$ for the weighted fractional Dirichlet problem \eqref{Eq:Problem} under the following assumption
\begin{align}
    \label{assumption:1}
    \textbf{Superlinear growth:} \quad 2 < p(s) < 2_s^*, \quad 
    p'(0) \in \l(0, \frac{4}{N}\r).
\end{align}
First, we show the asymptotic expansion of the weighted non-linear integral term in the energy functional $E_s$ which plays an essential role in the fibering map analysis of the problem \eqref{Eq:Problem}.
\begin{lemma}\label{lem:tayexp:weighterm}
  Let $s \in (0,1)$, $\beta> 0$, $p: (0,1) \to \R^+$ be a $C^1$ function such that $\lim_{s \to 0^+} p(s) = \beta$ and $a$ satisfies \eqref{a:regulrity-asym}. Then, for every $\phi \in H_0^s(\Omega)$ the following expansion (about $s=0$) holds:
   \[
   \into a(s,x)|\ph|^{p(s)} ~dx = \|\ph\|^\beta_{L^\beta(\Omega)}  + s \into \left(a'(0,x)+p'(0)\ln  
    \abs{\ph}\right)\abs{\ph}^\beta ~dx +o(s).
   \]
\end{lemma}
\begin{proof}
    Let $\ph \in C^\infty_c(\Omega),$ using Taylor's expansion (about $s=0$) for the function $a(\cdot,x)|\ph|^{p(\cdot)}$ and by \eqref{a:regulrity-asym}, we get
    \[
    \begin{split}
     a(s,x) \abs{\ph}^{p(s)} &= a(0,x)\abs{\ph}^{\beta} + s\l(\frac{d}{ds} {\bigg|_{s=0}}{a(s,x)\abs{\ph}^{p(s)}}\r) + o(s)\\
     & = \abs{\ph}^{\beta} + s {\l(a'(0,x)\abs{\ph}^{\beta}+ \abs{\ph}^{\beta}p'(0)\ln (|\ph|)\r)}+o(s)\\
     & = \abs{\ph}^{\beta} + s\abs{\ph}^{\beta} {\l(a'(0,x)+ p'(0)\ln (|\ph|)\r)} + o(s).
    \end{split}
    \]
Therefore,
    \[
    \begin{split}
      \into a(s,x)|\ph|^{p(s)} ~dx &= \into \l(\abs{\ph}^{\beta} + s\abs{\ph}^{\beta} {\l(a'(0,x)+ p'(0)\ln (|\ph|)\r)} + o(s)\r) ~dx \\
      & = \|\ph\|^\beta_{L^\beta(\Omega)}  + s \into \left(a'(0,x)+p'(0)\ln \abs{\ph}\right)\abs{\ph}^\beta ~dx +o(s).
    \end{split}
    \]
The result follows for a general $\ph \in H^s_0(\Omega)$ via density argument.
\end{proof}

\begin{lemma}
    \label{lem:log-conv}
    Let $(s_n)_{n\in \mathbb N}, (t_n)_{n\in \mathbb N} \subset (0,\frac{1}{4})$ be such that $\lim_{n\to\infty} s_n = 0$, \\ $\lim_{n\to\infty} t_n = 0$. Let $\beta$ be as in \eqref{assumption on beta} and $g: (0,\frac{1}{4})\times\Omega \to \R$ be such that 
    \begin{equation}
    \label{unif:conv}
    \sup_{n \in \mathbb{N}} \|g(s_n,x)\|_{L^{\beta(s_n)}(\Omega)} < + \infty \ \text{and} \ 
    \lim_{n \to \infty} g(s_n,\cdot) = g(0,\cdot) \quad \text{a.e. in} \ \Omega.
    \end{equation}
    Let $u_0 \in L^2(\Omega)$ and $(u_n)_{n\in\mathbb N}\subset L^2(\Omega)$ be such that $u_n \to u_0$ in $L^2(\Omega)$ as $n \to \infty$. Then, up to to a subsequence
    \[\lim_{n\to\infty} \into g(s_n,x)\ln (u_n^2) \abs{u_n}^{t_n}u_n\ph ~dx = \into g(0,x)\ln (u_0^2)u_0\ph ~dx \quad \text{for all} \ \ph \in C_c^\infty(\Omega).\]
\end{lemma}
\begin{proof} 

Since $\Omega$ is bounded, by dominated convergence theorem, we have
\[\lim_{n\to\infty} \int_{\{|u_n|< 1\}} g(s_n,x)\ln (u_n^2) \abs{u_n}^{t_n}u_n\ph ~dx = \int_{\{|u_0|< 1\}} g(0,x)\ln (u_0^2) u_0\ph ~dx.\]
For every $\delta>0$, there exists $c=c(\delta)$ (independent of $n$) such that
\[\ln (|u_n|^2) \leq c |u_n|^{2\delta} \quad  \text{for all} \ |u_n| \geq 1. \]
We choose $\delta$ such that $4s_n+2\delta + t_n <1$ for all $n \geq N_0.$ Since $u_n \to u$ in $L^2(\Omega)$, by \cite[Lemma A.1]{Willem}, there exists a $V\in L^2(\Omega)$ such that 
\[
    |u_n|< V \quad \text{for all} \ n\in \N, \ \text{a.e.} \ x  \in \Omega.
\]
Since $\beta(s) > \frac{2_s^\ast}{2_s^\ast-p(s)} > \frac{N}{2s} \geq \frac{1}{2s}$ and by the above choice of $\delta$, we have  
\[
\frac{2}{1-2\delta-t_n} < \beta(s_n).
\]
Thus, by applying Young's inequality, for all $ n\geq N_0$, for a.e. $x \in \Omega$, we get
\[
\begin{split}
   |g(s_n,x)\ln (u_n^2)\abs{u_n}^{t_n}u_n\ph|&\leq |g(s_n,x)||V|^{2\delta +1+t_n}|\ph|\\ 
   & \leq C(\|\ph\|_{L^\infty(\Omega)}) \l( |g(s_n,x)|^{\frac{2}{1-2\delta-t_n}}+ |V|^2\r)\\
   & \leq C(\|\ph\|_{L^\infty(\Omega)}) \l(1+ |g(s_n,x)|^{\beta(s_n)}+ |V|^2\r).
\end{split}
\]
Finally, by using Vitali convergence theorem and \eqref{unif:conv}, we have
\[\lim_{n \to \infty} \int_{\{|u_n| \geq  1\}} g(s_n,x)\ln (u_n^2) \abs{u_n}^{t_n}u_n\ph ~dx = \int_{\{|u_0| \geq 1\}} g(0,x)\ln (u_0^2) u_0\ph ~dx.\]
Hence the result.
\end{proof}
\begin{remark}
By \cite[eq. (3.4)]{Santamaria-Saldana-2022}, we have  \[
{\l(\kappa_{N,s}\r)}^{\frac{p(s)}{2(2-p(s))}}<+\infty \quad \text{for all} \  s \in (0,\frac{1}{4}).
\]
Furthermore, for all $s \in (0,\frac{1}{4})$, by \eqref{assumption on beta}, \ ${|\Omega|}^{\frac{1}{s\beta(s)}} <+\infty$.
Thus, \begin{equation}
\label{lower bound on s-norm}
    M:= \sup_{s \in (0,\frac{1}{4})} {\l(\|a(s,\cdot)\|_{L^{\beta(s)}(\Omega)} \  {|\Omega|}^{1-\frac{1}{\beta(s)}-\frac{p(s)}{2_s^\ast}} \ (\kappa_{N,s})^{\frac{p(s)}{2}}\r)}^{\frac{1}{2-p(s)}} < +\infty.
\end{equation}
\end{remark}
We now present a uniform lower bound for the solutions of the problem \eqref{Eq:Problem}. This estimate is crucial for demonstrating the non-triviality of the limit of the sequence of solutions \( u_s \) of \eqref{Eq:Problem}.
\begin{lemma}\label{lem:lowerbdd:NS}  
Let $p$ satisfies \eqref{assumption:1}, and $a$ satisfies satisfies \eqref{a:regulrity-asym} and \eqref{weifunc:bound}. Then, $u\in N_s$ and $\|u\|_s \geq M$ for all $s\in(0,\frac{1}{4}).$
    
\end{lemma}
\begin{proof}
    Let $J_s$: $H^s_0(\Omega)\setminus \{0\} \rightarrow\R$ be given by
    \[J_s(u)= \|u\|_s^2-\into a(s,x)|u|^{p(s)} ~dx.\]
By applying H\"older inequality and Theorem \ref{Frac_Sobo_ineq}, we obtain
\[
\begin{split}
    J_s(u) & \geq \|u\|_s^2- \|a(s, \cdot)\|_{L^{\beta(s)}(\Omega)} \bigg[\into |u|^{p(s)\beta'(s)} ~dx\bigg]^{\frac{1}{\beta'(s)}}\\
    &\geq \|u\|_s^2- \|a(s, \cdot)\|_{L^{\beta(s)}(\Omega)} |\Omega|^{1-\frac{1}{\beta(s)} - \frac{p(s)}{2_s^\ast }} \|u\|_{L^{2_s^\ast}(\Omega)}^{p(s)}\\
    & \geq \|u\|^2_s \left(1-\|a(s, \cdot)\|_{L^{\beta(s)}(\Omega)} |\Omega|^{1-\frac{1}{\beta(s)} - \frac{p(s)}{2_s^\ast }} (\kappa_{N,s})^{\frac{p(s)}{2}} \|u\|^{p(s)-2}_s\right).
\end{split}
\]
If $\|u\|_s \geq M$, we are done. Otherwise, if $\|u\|_s < M$, then from \eqref{lower bound on s-norm}, we have
\[
J_s(u) \geq \|u\|^2_s \left(1-\|a(s, \cdot)\|_{L^{\beta(s)}(\Omega)} |\Omega|^{1-\frac{1}{\beta(s)} - \frac{p(s)}{2_s^\ast }} \kappa_{N,s}^{\frac{p(s)}{2}} \|u\|^{p(s)-2}_s\right) > 0.
\]
Now, if $u\in N_s$ then $J_s(u)=0$, which contradicts the inequality. Hence, $\|u\|_s \geq M$, for all $s \in (0,\frac{1}{4})$ and $u \in N_s$.
\end{proof}

\noindent For a given $\ph\in H_0^s(\Omega) \setminus\{0\}$ such that $\into a(s,x)|\ph|^{p(s)} ~dx > 0$, we denote
\[r_{s,\ph} = \l(\frac{\|\ph\|_s^2}{ \into a(s,x)|\ph|^{p(s)} ~dx} \r)^{\frac{1}{p(s)-2}}.\]

The next result depicts the convergence of the elements in the  Nehari manifold $N_s$ as $s \to 0^+$. 
\begin{lemma}\label{lem:NS:identifi}
     Let $p$ satisfies \eqref{assumption:1} and $a$ satisfies \eqref{a:regulrity-asym} and \eqref{a:regulrity-1-asym}. Then, for any $\ph\in H_0^s(\Omega) \setminus\{0\}$ and $\into a(s,x)|\ph|^{p(s)} ~dx > 0$, $(r_{s,\ph}) \ph \in N_s$
and 
\[\lim_{s\to 0^+} r_{s,\ph} = r_{0,\ph} := \exp{\l(\frac{\mathcal{E}_L(\ph,\ph)-p'(0) \into \ln \abs{\ph}\ph^2 ~dx - \into a'(0,x)\abs{\ph}^2 ~dx}{p'(0)\|\ph\|^2_{L^2(\Omega)}}\r)} >0.\]
In particular, $\sup_{s\in [0,\frac{1}{4}]}r_{s,\ph}<\infty$.
\end{lemma}

\begin{proof}
By definition of $r_{s, \ph}$, $(r_{s, \ph}) \ph \in N_s.$ Now, by using \cite[Lemma 3.7]{Santamaria-Saldana-2022} and Lemma \ref{lem:tayexp:weighterm} with $\beta =2$, we obtain
\[
\begin{split}
    \lim_{s\to 0^+} r_{s,\ph} &= \lim_{s\to 0^+} {\l(\frac{\|\ph\|^2_{L^2(\Omega)}+s\mathcal{E}_L(\ph,\ph)+o(s)}{\|\ph\|^2_{L^2(\Omega)} +s \into (a'(0,x)+p'(0)\ln  
    \abs{\ph})\abs{\ph}^2 ~dx +o(s)}\r)}^{\frac{1}{p(s)-2}}.
\end{split}
\]
Setting $A=\|\ph\|^2_{L^2(\Omega)}$ and using \cite[Lemma 3.1]{Santamaria-Saldana-2022}, we get,
\[
\begin{split}
 \lim_{s\to 0^+} r_{s,\ph} &= \lim_{s\to0^+}{\l(\frac{{(1+sA^{-1}\mathcal{E}_L(\ph,\ph)+o(s))}^{\frac{1}{s}}}{{(1+\frac{s}{A} \into (a'(0,x)+p'(0)\ln \abs{\ph})\abs{\ph}^2 ~dx +o(s))}^{\frac{1}{s}}}\r)}^{\frac{s}{p(s)-2}} \\
&= {\l(\frac{\exp({A^{-1}\mathcal{E}_L(\ph,\ph)})}{\exp({A^{-1}\into a'(0,x)\abs{\ph}^2 ~dx})  .\exp({A^{-1}p'(0)\into \ln \abs{\ph} \ph^2 ~dx})}\r)}^{\frac{1}{p'(0)}}\\
&=  \exp{\l(\frac{\mathcal{E}_L(\ph,\ph)-p'(0)\into \ln \abs{\ph} \ph^2 ~dx - \into a'(0,x)\abs{\ph}^2 ~dx}{p'(0)\|\ph\|^2_{L^2(\Omega)}}\r)}  > 0. 
 \end{split}
 \]
This implies that the map $s \to r_{s, \ph}$ has a continuous extension on $[0,\frac{1}{4}]$. Hence $\sup_{s\in [0,\frac{1}{4}]} r_{s,\ph} < \infty.$
\end{proof}
Next, we show uniform upper estimates for the solutions $u_s$ of the problem \eqref{Eq:Problem}, for all $s \in (0, \frac{1}{4})$.
\begin{proposition}
     \label{uniform-bdd_s}
Let $s \in (0,\frac{1}{4})$, $p$ as in assumption \eqref{assumption:1}, $a$ satisfies \eqref{a:regulrity-asym}, \eqref{weifunc:bound} and \eqref{a:regulrity-1-asym}. Let $u_s \in N_s$ be a least energy solution of \eqref{Eq:Problem}, then, there is a constant $C=C(p,\Omega)>0$ such that 
\[
\begin{split}
\|u_s\|^2 = \mathcal{E}(u_s,u_s)<C \quad \text{for all} \ s\in (0,\frac{1}{4}).
\end{split}
\]
\end{proposition}
\begin{proof}
    Let $\ph\in H_0^s(\Omega)\setminus{\{0\}}$.
    By Lemma \ref{lem:NS:identifi} and using the fact that $u_s$ is least energy solution, we get
     \begin{equation} 
     \label{bound on s-norm}
    \|u_s\|_s^2 = \inf_{v\in N_s} \|v\|^2_s \leq (r_{s,\ph})^2\|\ph\|^2_s \leq \sup_{s\in(0,\frac{1}{4})} (r_{s,\ph})^2\|\ph\|^2_s =:C_0 < \infty.
    \end{equation}
Finally, by using $\mathcal{E}(u_s, u_s) \leq \mathcal{E}_s(u_s, u_s) = \|u_s\|_s^2$, we have the desired result.
\end{proof}
\subsection{Asymptotics for non-local superlinear problem}
In this subsection, we study the asymptotics for the non-local problem \eqref{Eq:Problem} with superlinear growth, {\it i.e.,} $p(\cdot)$ satisfies \eqref{assumption:1}. We begin with proving the existence of a solution for the problem \eqref{Eq:Problem} with suitable conditions on the weight function $a(s, \cdot)$.
\begin{theorem}
    \label{theorem-6.1}
Let $s\in (0,\frac{1}{4})$, $p(s)\in (2, 2^*_s)$ and $a: [0, \frac{1}{4}] \times \Omega \to \mathbb{R}^+$ be such that
\[
a(s, \cdot) \in L^{\frac{2^*_s}{2^*_s-p(s)}}(\Omega) \quad \text{for all} \ s \in [0, \frac{1}{4}].
\]
There exists a non-zero least energy weak solution $u_s$ of the problem \eqref{Eq:Problem}. Also, if $a(s,\cdot)$ is non-negative, for all $s\in [0, \frac{1}{4}]$, then, all least energy solutions of \eqref{Eq:Problem} are either positive or negative in $\Omega$.
\end{theorem}
\begin{proof}
    Let $\psi : H^s_0(\Omega)\setminus\{0\} \to \R$ be given by
    \[\psi(u) = \|u\|_s^2\ - \into a(s,x)\abs{u}^{p(s)} ~dx
    \]
 and 
 \[(\psi'(u), u)_{\mathbb{H}}= 2\|u\|^2_s-{p(s)}\into a(s,x)\abs{u}^{p(s)} ~dx.\]
Moreover,
\[
    (\psi'(u), u)_{\mathbb{H}} = 2\|u\|_s^2-p(s)\into a(s,x)\abs{u}^{p(s)} ~dx = \l(2-p(s)\r)\|u\|^2_s <0
    \]
which in turn implies $0$ is a regular value of $\psi$. Therefore, $N_s = \psi^{-1} (0)$ is a $C^1$ manifold.  Clearly, any minimizer $u$ of $E_s$ restricted to $N_s$ satisfies $\psi(u)=0,$ so we can apply the critical point theory on $N_s$ to get critical points of $E_s.$ By Ekeland's variational principle \cite[Corollary 3.4]{Ekeland} and Lemma \ref{lem:lowerbdd:NS}, there exists $(u_n)_{n\in \mathbb{N}} \subset N_s$, $(\zeta_n)_{n\in \mathbb{N}}\subset \R$, and $C>1$ such that  for all $n\in \mathbb N$, $C^{-1}\leq \|u_n\|_s\leq C,$
\begin{equation}
    \label{variationalprinciple-2}
   0\leq E_s(u_n)-\inf_{N_s} E_s \leq \frac{1}{n^2}, \quad \|E_s'(u_n)-\zeta_n \psi'(u_n)\|_{\mathcal{L}(H^s_0(\Omega),\R)}\leq \frac{1}{n}.
\end{equation}
Moreover, as $n \to \infty$
\[
\begin{split}
& o(1) = \frac{1}{\|u_n\|_s}\l(E'_s(u_n)u_n-\zeta_n \psi'(u_n)u_n \r)\\
& \ =\frac{1}{\|u_n\|_s}\l(-\zeta_n(2-p(s))\|u_n\|^2_s\r) = \zeta_n(p(s)-2)\|u_n\|_s.
\end{split}
\]
Therefore, $\zeta_n \to 0$ as $n \to \infty$. From above calculations we get, $\|E'_s(u_n)\|_{\mathcal{L}(H^s_0(\Omega),\R)} \to 0$ as $n\to \infty$ and there exists a $u_s \in H^s_0(\Omega)\setminus\{0\}$ such that on  passing to a subsequence, $u_n \rightharpoonup u_s$ in $H^s_0(\Omega)$ and consequently, $u_n\to u_s$ in $L^{p(s)}(\Omega)$ as $n\to \infty$. Now, for $\ph \in C^\infty_c(\Omega)$, we have
\begin{align}
\label{4.18 weak sol}
\begin{split}
 0 &= \lim_{n\to \infty} (E_s'(u_n), \ph)_{\mathbb{H}} = \lim_{n\to \infty} \left(\mathcal{E}_s(u_n,\ph)-\into a(s,x)\abs{u_n}^{p(s)-2}u_n \ph ~dx \right)\\
 &= \mathcal{E}_s(u_s,\ph)-\into a(s,x)\abs{u_s}^{p(s)-2}u_s \ph ~dx.
\end{split}
\end{align}
This leads us via density arguments to the conclusion that $u_s$ is a non-trivial solution of the problem \eqref{Eq:Problem}. Assume, $(\ph_n)_{n\in \mathbb N} \subset C^\infty_c(\Omega)$ such that $\ph_n \to u_s$ in $H_0^s(\Omega)$ as $n\to \infty$. By \cite[Theorem 1.1]{Cotsiolis-Tavoularis-2004} and \eqref{4.18 weak sol}, we get
 \[
 \begin{split}
 0 &= \lim_{n\to \infty} \l(\mathcal{E}_s(u_s,\ph_n) - \into a(s,x) |u_s|^{p(s)-2} u_s \ph ~dx \r)\\
 &=\mathcal{E}_s(u_s,u_s) - \into a(s,x) |u_s|^{p(s)} ~dx.
 \end{split}
 \]
 Therefore, $u_s \in N_s$. Since $u_n , u_s \in N_s$ and $u_n \to u_s$ in $L^{p(s)}(\Omega)$,  one has, up to a subsequence,
\[
\begin{split}
\inf_{N_s} E_s &= \lim_{n\to \infty} E_s(u_n) = \frac{1}{2} \lim_{n\to \infty} \|u_n\|^2_s-  \lim_{n\to \infty} \frac{1}{p(s)} \into a(s,x)\abs{u_n}^{p(s)} ~dx \\
&= \l(\frac{1}{2}- \frac{1}{p(s)}\r) \into a(s,x)\abs{u_s}^{p(s)}~dx = E_s(u_s)  \geq   \inf_{N_s} E_s.
 \end{split}
 \]
 Let $u_s$ be the least energy solution and $r_{s,\abs{u_s}}$ be defined as in Lemma \ref{lem:NS:identifi}. Then, $\| \abs{u_s} \|_s \leq \|u_s\|_s$ , $r_{s,\abs{u_s}}\leq 1$, and $r_{s,\abs{u_s}}\abs{u_s}\in N_s$. Now,
 \[E_s(u_s) \leq E_s(r_{s,\abs{u_s}}\abs{u_s}) \leq E_s(u_s),\]
 yielding $r_{s,\abs{u_s}} =1$ and $\abs{u_s}$ is a non-negative least energy solution of \eqref{Eq:Problem}. By the strong maximum principle (see \cite[Theorem 1.2]{Pezzo-Quaas-2017}), $\abs{u_s}>0$ in $\Omega$. Thus, we can conclude that $u_s$ is either strictly positive or strictly negative in $\Omega$.
 \end{proof}
 
\noindent \textit{Proof of Theorem \ref{thm:asym-super} :} 
 Let $\{u_{s}\} \subset N_s$ be a sequence of least energy solutions of the  problem \eqref{Eq:Problem}. The existence of such a sequence is given by Theorem \ref{theorem-6.1}.
 By \eqref{bound on s-norm} and \cite[Lemma 3.6]{Santamaria-Saldana-2022},
 \[\|u_{s}\|_{s} \leq C_0 \ \text {for some} \ C_0 > 0 \  \text{depending on} \ \Omega,\ p.\]
  Thus, by Proposition \ref{uniform-bdd_s}, $\{u_{s}\}$ is uniformly bounded in $\h(\Omega)$. Therefore, by the compact embedding\ $ \h(\Omega) \hookrightarrow L^2(\Omega),$ we have
 \[ u_{s} \rightharpoonup u_0 \  \text{in} \  \h(\Omega), \quad  u_{s} \to u_0 \ \text{in} \ L^2(\Omega) \  \text{as} \ s \to 0^+. \]
 Now, using \cite[Theorem 1.1]{Chen-Weth-2019}, \eqref{a:regulrity-asym} and assuming $\ph \in C_c^\infty(\Omega)$, we obtain
 \[
 \begin{split}
     \into & u_{s} (\ph + s L_\Delta \ph +o(s)) ~dx= \into u_{s}(-\Delta)^{s} \ph ~dx = \into a(s,x) \abs{u_{s}}^{p(s)-2} u_{s}\ph ~dx\\ 
     &=\into u_{s} \bigg(1+ s \int_0^1 \bigg[ \big[a'(s \tau,x) +a(s \tau,x)\ln \abs{u_{s}} p'(s\tau)\big]\abs{u_{s}}^{p(s \tau)-2}\bigg] ~d\tau+o(s) \bigg)\ph ~dx
 \end{split}
 \]
 in $L^\infty(\Omega).$ Thus, by \cite[eq (3.11)]{Chen-Weth-2019}, we have
 \[
 \begin{split}
 \mathcal{E}_L(u_{s}, \ph) & + o(1) = \into u_{s} L_\Delta \ph ~dx +o(1)\\
 &= \into \ph \int_0^1 \bigg[ \big[a'(s \tau,x)+a(s\tau,x)\ln \abs{u_{s}} p'(s\tau)\big]\abs{u_{s}}^{p(s \tau )-2} \bigg]u_{s} ~d\tau  ~dx
 \end{split}
 \]
 as $s \to 0^+$, for all $\ph \in C^\infty_c(\Omega).$ Using \eqref{a:regulrity-asym}, Lemma \ref{lem:log-conv}, passing to a subsequence and letting $s \to 0^+$, we obtain
 \[\mathcal{E}_L(u_0, \ph)= \into  \bigg[a'(0,x)+p'(0)\ln \abs{u_0}\bigg]u_0 \ph ~dx \quad \text{for all} \ \ph \in C^\infty_c(\Omega).\]
 Therefore, by density arguments, we can conclude that $u_0$ is a weak solution of the problem \eqref{Eq:LimitingProblem}. 
By using $\|u_{s}\|_{s} \leq C_0$, Lemma \ref{lem:lowerbdd:NS}, H\"older inequality and Theorem \ref{Frac_Sobo_ineq}, we get
 \begin{equation}\label{est:non-trivial-1}
     M^2 \leq \|u_{s}\|_{s}^2 \leq \|a(s,x)\|_{L^{\beta(s)}(\Omega)} {\bigg[\into |u_{s}|^{p(s)\beta'(s)} ~dx \bigg] }^{\frac{1}{\beta'(s)}}
 \end{equation}
 and 
 \begin{equation}\label{est:non-trivial-2}
     \begin{split}
     \into |u_{s}|^{p(s)\beta'(s)} ~dx & = \into |u_{s}|^{ \alpha_k} |u_{s}|^{\beta_k} ~dx\\
     & \leq \l(\into |u_{s}|^2 ~dx \r)^{1-\lambda_s} \l(\into |u_{s}|^{2_{s}^\ast} ~dx \r)^{\lambda_s}\\
     & \leq \|u_{s}\|_{L^2(\Omega)}^{2(1-\lambda_s)} k_{N,s}^{\frac{\lambda_s 2_{s}^\ast}{2}} \|u_{s}\|_{s}^{\lambda_s 2_{s}^\ast}  \leq \|u_{s}\|_{L^2(\Omega)}^{2(1-\lambda_s)} k_{N,s}^{\frac{\lambda_s 2_{s}^\ast}{2}} C_0^{\lambda_s 2_{s}^\ast},
      \end{split}
 \end{equation}
 where 
 \[
\lambda_s:= \frac{p(s) \beta'(s)-2}{2_{s}^\ast-2} , \quad \alpha_s = (1-\lambda_s)2, \ \beta_s = \lambda_s 2_{s}^\ast \quad \alpha_s + \beta_s = p(s)\beta'(s).
 \]
 Combining \eqref{est:non-trivial-1} and \eqref{est:non-trivial-2}, we get
 \[
 \begin{split}
      \frac{M^{2 \beta'(s)}}{\|a(s,x)\|_{L^{\beta(s)}(\Omega)}^{\beta'(s)} k_{N,s}^{\frac{\lambda_s 2_{s}^\ast}{2}}}
      & \leq \|u_{s}\|_{L^2(\Omega)}^{2(1-\lambda_s)} C_0^{\lambda_s 2_{s}^\ast}
 \end{split}
 \]
which further implies that
 \[
 \begin{split}
     \l(\frac{M^{2 \beta'(s)}}{\|a(s,x)\|_{L^{\beta(s)}(\Omega)}^{\beta'(s)} k_{N,s}^{\frac{\lambda_s 2_{s}^\ast}{2}}} \r)^\frac{1}{2(1-\lambda_s)} C_0^\frac{-\lambda_s 2_{s}^\ast}{2(1-\lambda_s)} & \leq  \|u_{s}\|_{L^2(\Omega)}.
 \end{split}
 \]
Now, we claim that $\lim_{s \to 0^+} \lambda_s = \lambda \in (0,1).$
Using \eqref{assumption on beta}, we have
\[
\begin{split}
   \frac{\beta'(s)(p(s)-2)}{2_{s}^\ast-2} < \lambda_s &= \frac{p(s)\beta'(s)-2}{2^\ast_{s}-2}
    = \frac{\beta'(s)(p(s)-2)}{2_{s}^\ast-2}+ \frac{2(\beta'(s)-1)}{2^\ast_{s}-2}\\
    &= \l(\frac{p(s)-2}{s}\r)\l(\frac{N-2s}{4}\r)\l(\frac{1}{1-\frac{1}{\beta(s)}}\r)+ \frac{N-2s}{2s(\beta(s)-1)}\\
    &< \l(\frac{p(s)-2}{s}\r)\l(\frac{N-2s}{4}\r)\l(\frac{1}{1-\frac{1}{\beta(s)}}\r)+1-\frac{Np'(0)}{4}-\gamma.
\end{split}
\]
Thus, by passing limit $s \to 0^+$ in above inequality, we obtain $0 < \lim_{s \to 0^+} \lambda_s = \lambda \leq (1-\gamma) <1.$ Hence the claim follows. We now proceed to evaluate the following limit \[\lim_{s \to 0^+} \l(\frac{M^{2 \beta'(s)}}{\|a(s,x)\|_{L^{\beta(s)}(\Omega)}^{\beta'(s)} k_{N,s}^{\frac{\lambda_s 2_{s}^\ast}{2}}} \r)^\frac{1}{2(1-\lambda_s)} C_0^\frac{-\lambda_s 2_{s}^\ast}{2(1-\lambda_s)}=:L.\]
Note that \[\begin{split}
L &= \exp \bigg(\lim_{s \to 0^+} \bigg[\frac{\beta'(s)}{(1-\lambda_s)} \ln M - \frac{\beta'(s)}{2(1-\lambda_s)} \ln \|a(s,x)\|_{\beta(s)}\\& \qquad \qquad \qquad \qquad \qquad-\frac{\lambda_s 2_{s}^\ast}{2(1-\lambda_s)}\l(\frac{\ln (\kappa_{N,s})}{2}+\ln C_0\r)\bigg] \bigg),
\end{split}\]
and 
\[ \beta(s) \to \infty, \ \beta'(s) \to 1, \ p(s) \to 2, \ 2_{s}^\ast \to 2, \ \kappa_{N,s} \to 1, \ \text{as} \  s \to 0^+. \]
It is clear that
\[
\lim_{s \to 0^+} \frac{\beta'(s)}{(1-\lambda_s)} \ln M = \frac{\ln M}{(1-\lambda)}, \quad \lim_{s \to 0^+} \frac{\lambda_s 2_{s}^\ast}{4(1-\lambda_s)}\ln (\kappa_{N,s}) = 0, \] \[ \text{and} \quad \lim_{s \to 0^+} \frac{\lambda_s 2_{s}^\ast}{2(1-\lambda_s)} \ln C_0 = \frac{\lambda \ln C_0}{(1-\lambda)}.\]
Now, by using \eqref{weifunc:bound}, for some $m>0$ (independent of $s$), we have
\[\|a(s,x)\|_{L^{\beta(s)}(\Omega)}^{\frac{1}{p(s)-2}} \leq m\] which in turn implies \[q:=\lim_{s \to 0^+} \|a(s,x)\|_{L^{\beta(s)}(\Omega)} \leq 1, \] and therefore, \[\lim_{s \to 0^+} \frac{\beta'(s)}{2(1-\lambda_s)} \ln \|a(s,x)\|_{L^{\beta(s)}(\Omega)} = \frac{\ln q}{2(1-\lambda)}.\]
From the above, we get
\[L= M^{\frac{1}{(1-\lambda)}} C_0^{\frac{\lambda}{(\lambda-1)}}q^{\frac{1}{2(\lambda-1)}}>0.\]
Thus,
\[0<L\leq\|u_0\|_{L^2(\Omega)}.\]
Hence, $u_0$ is non-trivial, $u_0 \in N_{0, p'(0)}$ and
 \[
 \mathbb{E}(u_0) = \frac{p'(0)}{4} \|u_0\|_{L^2(\Omega)}^2,
 \]
 where 
 \[
N_{0, p'(0)}:= \{u\in \h(\Omega)\setminus\{0\}: \mathcal{E}_L(u,u)=\into (a'(0,x)+ p'(0) \ln\abs{u})u^2 ~dx\}. 
\]
Now we want to show that $u_0$ has least energy, {\it i.e.,} $\mathbb{E}(u_0) = \inf_{N_{0, p'(0)}} \mathbb{E}.$ 
By Fatou's lemma, we have
     \begin{align}
     \label{Es/s}
 \begin{split}
 \inf_{N_{0, p'(0)}} \mathbb{E}=\frac{p'(0)}{4} \|u_0\|^2_{L^2(\Omega)}
 &\leq \frac{p'(0)}{4} \lim_{s \to 0^+} \inf \int_{\R^N} \abs{\xi}^{2s}\abs{\widehat{u_{s}}}^2 ~d\xi\\
 &= \lim_{s \to 0^+} \inf \frac{E_{s}(u_{s})}{s}= \lim_{s \to 0^+} \frac{E_{s}(u_{s})}{s} .
 \end{split}
 \end{align}
 It remains to show \[\lim_{s \to 0^+} \frac{E_{s}(u_{s})}{s}\leq \inf_{N_{0, p'(0)}}  \mathbb{E}.\]
 By Theorem \ref{Main-res-limitingprob}, there exists a $v \in {N_{0, p'(0)}} $ such that $\mathbb{E}(v)=\inf_{N_{0, p'(0)}}  \mathbb{E}$. Let $(v_n)_{n\in \mathbb N} \subset C^\infty_c(\Omega) \cap N_{0, p'(0)} $ such that $v_n\to v$ in $\h(\Omega).$ Then,
 \begin{equation}
 \label{v_n}
     \lim_{n \to \infty} \frac{p'(0)}{4}\|v_n\|^2_{L^2(\Omega)} = \frac{p'(0)}{4}\|v\|^2_{L^2(\Omega)} = \mathbb{E}(v) = \inf_{N_{0, p'(0)}}  \mathbb{E} . \end{equation}
 By  Lemma \ref{lem:NS:identifi}, we have $(r_{{s},{v_n}})v_n \in N_s$. Using the fact that $v_n \in N_{0, p'(0)} $, we get $\lim_{s \to 0^+} r_{{s},{v_n}} =1.$ Now, by applying \cite[Lemma 3.7]{Santamaria-Saldana-2022}, \eqref{v_n} and using $u_{s}$ is a least energy solution, we get
  \[
 \begin{split}
     \lim_{s \to 0^+}  \frac{E_{s}(u_{s})}{s} &\leq \lim_{s \to 0^+} \frac{1}{s} E_{s}((r_{{s},{v_n}})\ v_n) = \frac{p'(0)}{4} \|v_n\|^2_{L^2(\Omega)}.
     \end{split}
     \]
Passing to the limit as $n \to \infty$ and using \eqref{Es/s}, we obtain
   \[\inf_{N_{0, p'(0)}}  \mathbb{E} = \lim_{s \to 0^+} \frac{E_{s}(u_{s})}{s}. \]
Since $u_{s} \rightharpoonup u_0$ in $\h(\Omega),$ $u_{s} \to u_0$ in $L^2(\Omega)$ as $s \to 0^+$, from the above calculations, we get
\[
\begin{split}
\inf_{N_{0, p'(0)}}  \mathbb{E} \leq \mathbb{E}(u_0)= \frac{p'(0)}{4}\|u_0\|^2_{L^2(\Omega)} &\leq \frac{p'(0)}{4} \lim_{s \to 0^+} \inf \|u_{s}\|_{s}^2\\ &= \lim_{s \to 0^+} \inf \frac{E_{s}(u_{s})}{s}= \lim_{s \to 0^+} \frac{E_{s}(u_{s})}{s} = \inf_{N_{0, p'(0)}}  \mathbb{E}.
\end{split}
\]
Hence the desired claim follows. \qed

\section{Logistic type problem involving the logarithmic Laplacian}\label{logistic-results}
\subsection{Existence results for \eqref{prob:logarithmic-laplace} with $\sigma \in (-\infty,0)$}
In this subsection, we give the existence result for the problem \eqref{prob:logarithmic-laplace} with $\sigma \in (-\infty,0).$
We begin with showing the coercivity of the energy functional $\mathbb{E}.$
\begin{lemma}
      \label{coerc:prop}
      Let $f$ satisfies \eqref{assump f1}- \eqref{assump f3} and $\sigma \in (-\infty,0).$ Then, the energy functional $\mathbb{E}$ is coercive, {\it i.e.,}
      \[\lim_{\|u\|\to \infty , u\in \h(\Omega)} \mathbb{E}(u) = \infty.\]
  \end{lemma}

  \begin{proof}
By using \eqref{assump F1} { for $\eps < \frac{-\sigma}{2}$} and the definition of the quadratic form $\mathcal{E}_L(\cdot,\cdot)$, we get
       \[
       \begin{split}
       \mathbb{E}(u) &\geq \frac{1}{2}\mathcal{E}_L(u,u) - \left(C - \frac{\sigma}{4} \right) \|u\|_{L^2(\Omega)}^2  - { \l(\frac{\sigma}{2} +\eps\r)\into u^2 \ln |u| ~dx}\\
         & \qquad \qquad + { 2 \eps \into u^2\ln|u| {\bf 1}_{|u| \leq 1} ~dx}\\
        & \geq \frac{1}{2}\mathcal{E}(u,u) - \left(c_1 - \frac{\sigma}{4} \right) \|u\|_{L^2(\Omega)}^2  - { \l(\frac{\sigma}{2} +\eps\r)\into u^2 \ln |u| ~dx - \frac{\eps|\Omega|}{e}}.
       \end{split}
       \]  
where $c_1>0$ is a constant depending on $C$ defined in \eqref{assump F1}, { $\eps$ and $\Omega$. Denote 
\[ \Omega_1:= \l\{x \in \Omega: \ln(u^2(x)) \geq \frac{-4}{(\sigma+2\eps)} \left(c_1 - \frac{\sigma}{4} \right)\r\}.\] 
This implies
\[- \l(\frac{\sigma+2\eps}{4}\r)\int_{\Omega_1} |u|^2 \ln(|u|^2) ~dx \geq \left(c_1 - \frac{\sigma}{4} \right) \int_{\Omega_1} u^2 ~dx.\]
Therefore, we have
\[
\begin{split}
 \mathbb{E}(u) \geq \frac{1}{2}\mathcal{E}(u,u) &- \left(c_1- \frac{\sigma}{4} \right) \int_{\Omega\setminus{\Omega_1}} u^2 ~dx - \l(\frac{\sigma + 2\eps}{4}\r)\int_{\Omega\setminus{\Omega_1}} |u|^2 \ln(|u|^2) ~dx - \frac{\eps|\Omega|}{e}.
\end{split}
\]
Since $u^2(x) \leq \exp \l(\frac{-4}{(\sigma+2\eps)} \left(c_1 - \frac{\sigma}{4} \right)\r)$ in $\Omega\setminus \Omega_1$, there exists a constant $C_2 >0$ such that 
\[
\begin{split}
&\bigg[- \left(c_1 - \frac{\sigma}{4} \right)\int_{\Omega\setminus{\Omega_1}} u^2 ~dx - \l(\frac{\sigma+2\eps}{4}\r)\int_{\Omega\setminus{\Omega_1}} |u|^2 \ln(|u|^2) ~dx\bigg]\geq -C_2
\end{split}
\]}
and $\mathbb{E}(u) \geq \frac{1}{2} \mathcal{E}(u,u) - C_2$, which gives the desired claim.
  \end{proof}
\noindent \textit{Proof of Theorem \ref{them:existence-limiting}:}
By Lemma \ref{coerc:prop}, there is a minimizing sequence $\{u_k\}_{k \in \mathbb{N}}$ for the energy $\mathbb{E}$ such that
\[
m:= \inf_{u \in \h(\Omega)} \mathbb{E}(u) = \lim_{k \to \infty} \mathbb{E}(u_k).
\]
By compactness of the embedding $\h(\Omega)\hookrightarrow L^2(\Omega)$, there exists $u_0\in \h(\Omega)$ such that up to a subsequence, $u_k \rightharpoonup u_0$ in $\h(\Omega)$ and $u_k \to u_0$ in $L^2(\Omega)$ and a.e. in $\Omega$, as $k\to \infty$. Since the function $t \to t^2 \ln t^2$ is bounded below, using \eqref{assump F1}, by Fatou's lemma and Vitali convergence theorem, we have
\[
\int_{\Omega} u_0^2 \ln u_0^2 ~dx  \leq \liminf_{k \to \infty} \int_{\Omega} u_k^2 \ln u_k^2 ~dx, \quad \int_{\Omega} F(x,u_0) ~dx = \lim_{k \to \infty} \into F(x,u_k) ~dx
\]
and
\[
\lim_{k \to \infty} \int \int_{x, y \in \mathbb{R}^N, |x-y| \geq 1} \frac{u_k(x) u_k(y)}{|x-y|^N} ~dx ~dy = \int \int_{x, y \in \mathbb{R}^N, |x-y| \geq 1} \frac{u_0(x) u_0(y)}{|x-y|^N} ~dx ~dy.
\]
As a consequence, we have $\mathbb{E}(u_0) \leq \lim \inf_{k \to \infty} \mathbb{E}(u_k)=m$ and $u_0$ is the least energy solution of the problem \eqref{prob:logarithmic-laplace}. Now, we proceed to show the non-triviality of $u_0$. Let $\ph \in C^\infty_c(\Omega)\setminus \{0\}$. Then, using \eqref{assump F1} with $\eps < \frac{-\sigma}{2}$, we have
{  \[
  \begin{split}
   \mathbb{E}(u_0) & = m \leq \mathbb{E}(t \ph) \\
   &\leq t^2\bigg(\frac{1}{2} \mathcal{E}_L(\ph,\ph)-\bigg[\l(\frac{\sigma}{2} + \eps\r) \ln t  -\frac{\sigma}{4}-C\bigg] \|\ph\|^2_{L^2(\Omega)} \\
    & \qquad \qquad \qquad \qquad - \frac{\sigma}{2} \into \ph^2 \ln|\ph| ~dx +  \eps \into \ph^2 |\ln|\ph|| ~dx\bigg)< 0,
   \end{split}
   \]}
   for $t>0$ sufficiently small. Therefore $u_0 \neq 0.$ Moreover, if $F(x,t)\leq F(x,|t|)$ for all $t \in \mathbb{R}$,
we can conclude that $u_0$ has a constant sign in $\Omega$ by following the same arguments as in the proof of Theorem \ref{Main-res-limitingprob}. 
\qed
\subsection{D\'iaz-Saa type inequality}
Denote 
\[
V_+^q = \{u : \Omega \to (0,\infty) \ | \ u^{\frac{1}{q}} \in \h(\Omega)\}, \ \text{for} \ q \in (1,2].
\]
By \cite[Proposition 2.6]{Brasco-Franzina-2014}, the set $V_+^q$ is a convex cone, {\it i.e.,} for $\sigma\in (0,\infty)$ and $ \ f,g \in V^q_+$ we have $(\sigma f +g)\in V^q_+.$ Recall that by \cite[Proposition 3.2]{Chen-Weth-2019}, we have
\begin{equation}\label{equi:def}
    \mathcal{E}_L(u,u) = \frac{c_N}{2} \intoo \frac{|u(x)-u(y)|^2}{|x-y|^N} ~dx ~dy + \into [h_\Omega(x) + \rho_N] u^2 ~dx,
\end{equation}
where $$h_\Omega(x) = c_N \left(\int_{B_1(x)\setminus \Omega} \frac{1}{|x-y|^N} ~dy - \int_{\Omega \setminus B_1(x)} \frac{1}{|x-y|^N} ~dy\right).$$
\begin{proposition}
\label{convexity of W}
    The functional $W_q: V_+^q \to \mathbb{R}$ defined by 
    \[
    W_q(u) := \mathcal{E}_L(u^{\frac{1}{q}},u^{\frac{1}{q}})
    \]
is  
convex in $V_+^q$ for $q=2$. Moreover, if  $h_\Omega(x) + \rho_N \geq 0$ in $\Omega$, then $W_q$ is also 
convex in $V_+^q$ for $q \in (1,2).$
\end{proposition}
\begin{proof}
Let $w_1, w_2 \in V_+^q$ such that $w_1^{\frac{1}{q}}, w_2^{\frac{1}{q}} \in \h(\Omega)$ and set $u=tw_1+(1-t)w_2$, $t \in (0,1).$ By \cite[Proposition 4.1]{Brasco-Franzina-2014}, \eqref{equi:def} and convexity of the map $t \mapsto t^\frac{2}{q}$, we obtain
\[
\begin{split}
    W_q& (tw_1+(1-t)w_2) = W_q(u) \\
    & = \frac{c_N}{2} \intoo \frac{|u^\frac{1}{q}(x)-u^\frac{1}{q}(y)|^2}{|x-y|^N} ~dx ~dy + \into [h_\Omega(x) + \rho_N] u^\frac{2}{q} ~dx\\
    & \leq t \l(\frac{c_N}{2} \intoo \frac{|w_1^\frac{1}{q}(x)-w_1^\frac{1}{q}(y)|^2}{|x-y|^N} ~dx ~dy + \into [h_\Omega(x) + \rho_N] w_1^\frac{2}{q} ~dx\r) \\
    & \qquad + (1-t) \l(\frac{c_N}{2} \intoo \frac{|w_2^\frac{1}{q}(x)-w_2^\frac{1}{q}(y)|^2}{|x-y|^N} ~dx ~dy + \into [h_\Omega(x) + \rho_N] w_2^\frac{2}{q} ~dx\r)\\
    & = t W_q(w_1) + (1-t) W_q(w_2).
\end{split}
\]
\end{proof}

\noindent \textit{Proof of Theorem \ref{thm:modi-Diaz-Saa-inequ}:}
 For $w_1, w_2 \in \h(\Omega)$ such that $w_1>0, w_2>0$ a.e. in $\Omega$ and $\frac{w_1}{w_2}, \frac{w_2}{w_1} \in L^\infty(\Omega)$. Then, by \cite[Proposition 4.1]{Brasco-Franzina-2014} and for all $\theta \in [0, 1],$
\[v_\theta:=\l((1-\theta)w_1^q+\theta w_2^q\r) > 0 \quad \text{and} \quad v_\theta^{\frac{1}{q}} \in \h(\Omega),\]
and 
\begin{equation}\label{testfunc:imbed}
    \frac{w_1^q}{v_\theta^{\frac{q-1}{q}}}, \frac{w_2^q}{v_\theta^{\frac{q-1}{q}}} \in \h(\Omega) \cap L^\infty(\Omega).
\end{equation}
Define $\Phi:[0,1] \to \R$ by 
\[ \Phi(\theta) = \mathcal{E}_L(v_\theta^\frac{1}{q}, v_\theta^\frac{1}{q}).\]
Then, by Proposition \ref{convexity of W}, $\Phi$ is convex and 
\[
\begin{split}
\Phi(\theta)=\mathcal{E}_L(v_\theta^{\frac{1}{q}}, v_\theta^{\frac{1}{q}}) &= \|v_\theta^\frac{1}{q}\|^2 -c_N \intb \frac{v_\theta^{\frac{1}{q}}(x)v_\theta^{\frac{1}{q}}(y)}{|x-y|^N} ~dx ~dy\\
& \quad +\rho_N \int_{\R^N} v_\theta^{\frac{1}{q}}(x) v_\theta^{\frac{1}{q}}(x) ~dx.
\end{split}
\]
Now, on differentiating with respect to $\theta$ and using \eqref{testfunc:imbed}, we obtain
\begin{equation}
    \begin{split}
        \frac{q}{2} \Phi'(\theta) & = \mathcal{E}\l(v_\theta^{\frac{1}{q}}, (w_2^q-w^q_1) v_\theta^\frac{-(q-1)}{q}\r) + \rho_N \into v_\theta^{\frac{2}{q}-1}(x)(w_2^q-w_1^q)(x) ~dx. \\
        & \quad - c_N \intb \frac{v_\theta^{\frac{1}{q}}(y)(w_2^q-w^q_1)(x)v_\theta^\frac{-(q-1)}{q}(x)}{|x-y|^N} ~dx ~dy\\
        & = \mathcal{E}_L\l(v_\theta^{\frac{1}{q}}, \frac{w_2^q-w^q_1}{v_\theta^\frac{(q-1)}{q}}\r).
    \end{split}
\end{equation}
Finally, by using the convexity of $\Phi$ we get
\[\Phi'(1):= \lim_{\theta \to 1^-} \Phi'(\theta) \geq \lim_{\theta \to 0^+} \Phi'(\theta):= \Phi'(0)\]
and hence the claim for $q \in (1,2]$. By \cite[Theorem 1.8 and Corollary 1.9]{Chen-Weth-2019} for $q=1$, we have
\[
\begin{split}
 \mathcal{E}_L\l(w_2, w_2-w_1\r) &- \mathcal{E}_L\l(w_1, w_2-w_1\r) = \mathcal{E}_L\l(w_2-w_1, w_2-w_1\r) \\
 & \geq \lambda_{1,L} \|w_2-w_1\|_{L^2(\Omega)}^2 \geq 0 \quad \text{if} \ h_\Omega + \rho_N \geq 0 \ \text{in} \ \Omega.
 \end{split}
\]
\qed
\subsection{Regularity and Uniqueness results} In this subsection, we prove the uniqueness and regularity results of solutions to the problem \eqref{prob:logarithmic-laplace} in the case $\sigma \in (-\infty,0).$

\noindent \textit{Proof of Theorem \ref{thm:regularity-uniqueness}:} 
 First, we show that $u \in L^\infty(\Omega).$ 
   Let $0\leq u, v \in \h(\Omega)$. Using \eqref{assumption2} with $\eps< -\sigma$ and the fact that $|t| \ln |t|$ is bounded below by $\frac{-1}{e}$, we obtain
  \[
  \begin{split}
      \mathcal{E}_{L}(u,v) &= \into ( f(x,u) + \sigma \ln\abs{u}u )v ~dx \\
      & \leq C_0 \int_\Omega v ~dx + {\int_{\Omega} u v \l( \sigma \ln |u| + \eps |\ln|u||\r) ~dx \leq \l(C_0-\frac{(\sigma-\eps)}{e} \r)\int_\Omega v ~dx}. 
  \end{split}
  \]
Then, by \cite[Theorem 6.7]{Dyda-Jarohs-Sk-2025}, $u\in L^\infty(\Omega).$
  Using \eqref{assumption2}, we have $f(x,u) + \sigma u \ln|u| \in L^\infty(\Omega).$
  Since $u\geq 0$, \cite[Theorem 1.1]{Santamaria-Rios-Saldana-2024}, gives $u \in C(\mathbb{R}^N).$ Now, we argue by contradiction to prove that $u>0$ in $\Omega.$ Assume that there is a $x_0 \in \Omega$ such that \begin{equation} \label{eq}u(x_0)=0.\end{equation} Using $u \in C(\R^N)$ and its non-triviality, there exists a $\delta>0$, an open set $V \subset \{x \in \Omega \ | \ u(x)>\delta\}$ and $r>0$ (depending on $\theta$ in \eqref{f-est-near-0}) such that (see \cite[Corollary 1.9 (ii)]{Chen-Weth-2019})
\[
|B_r(x_0)| \leq 2^N \exp\l(\frac{N}{2} \l(\psi(\frac{N}{2})-\gamma\r)\r) |B_1(0)|
\]
and by \eqref{f-est-near-0}, we have
     \begin{equation}
     \label{eq:bound}
     f(x,u) + \sigma \ln (u)u \geq 0 \quad \text{in} \ B_r(x_0), \quad  \dist(B_r(x_0),V)>0.
     \end{equation}
     With the above choice of $r$, $L_\Delta$ satisfies weak maximum principle in $B_r(x_0)$ and thus by \cite[Theorem 4.8]{Chen-Weth-2019}, $\lambda_{1,L}(B_r(x_0))>0.$
     By \cite[Theorem 1.1]{Lara-Saldana-2022}, there exists a unique classical solution $\tau \in \h(\Omega)$ to the problem
     \[L_\Delta \tau =1 \quad \text{in} \ B_r(x_0), \quad \tau=0 \quad \text{in} \ \R^N\setminus B_r(x_0).\]
     Thus, $L_\Delta \tau(x) =1$ holds pointwise for a.e. $x\in\Omega,$ which in turn implies $\tau>0$ in $B_r(x_0).$ Let $\chi_V$ denotes the characteristic function of $V$, then, for $x \in B_r(x_0)$, $\chi_V(x)=0.$ Let $K:= c_N|V| \inf_{z \in B_r(x_0)} (|z-y|^{-N})$ and $\ph:= \frac{K}{2}\tau + \chi_V.$ By \cite[Theorem 1.1]{Chen-Weth-2019}, we have
     \[
\begin{split}
    L_\Delta \chi_V(x) &= -c_N \int_{\R^N} \frac{\chi_V(y)}{|x-y|^N} ~dy = - c_N \int_V \frac{1}{|x-y|^N} ~dy \leq -K
   \end{split}
     \]
which further implies
    \[L_\Delta \ph \leq \frac{K}{2}-K \leq 0 \quad\text{in} \ B_r(x_0). \]
Using \eqref{eq:bound}, we get
    \[L_\Delta (u-\delta \ph ) = L_\Delta u - \delta L_\Delta \ph \geq L_\Delta u +\frac{\delta K}{2} \geq 0 \quad \text{in} \ B_r(x_0) \]
    and \[(u-\delta \ph) \geq 0 \quad \text{in}\ \R^N\setminus B_r(x_0).\]
Now, by applying the weak maximum principle for $L_\Delta$, we obtain 
\[
u \geq \delta \ph \geq \delta \tau >0 \quad \text{in} \  B_r(x_0)
\] which contradicts \eqref{eq} Thus, $u>0$ in $\Omega.$ Finally, by \cite[Corollary 5.3]{Santamaria-Rios-Saldana-2024}, we have
\[C^{-1} \ell^\frac{1}{2}(\delta(x)) \leq u(x) \leq C \ell^\frac{1}{2}(\delta(x)) \quad \text{for all} \ x \in \Omega \ \text{and} \ \text{for some} \ C>0.\] 
Next, we prove the uniqueness part. Let $w_1$ and $w_2$ be two non-trivial and non-negative weak solutions of the problem \eqref{prob:logarithmic-laplace} such that the set $\Omega_0 :=\{x \in \Omega: w_1(x) \neq w_2(x)\}$ is non-empty and $|\Omega_0| \neq 0$. Then, $w_1, w_2 \in C(\mathbb{R}^N)$ and 
      \[
      C^{-1} \ell^\frac{1}{2}(\delta(x)) \leq w_1(x), w_2(x) \leq C \ell^\frac{1}{2}(\delta(x)) \quad \text{for all} \ x \in \Omega.
      \]
This further gives $\frac{w_1}{w_2}, \frac{w_2}{w_1} \in L^\infty(\Omega)$ and 
\[
\phi:= \frac{w_2^q - w_1^q}{w_1^{q-1}} \in \h(\Omega) \setminus \{0\}, \quad \psi:= \frac{w_2^q-w_1^q}{w_2^{q-1}} \in \h(\Omega)\setminus \{0\}.
\]  
Now, by taking $\phi$ and $\psi$ as test functions, using Theorem \ref{thm:modi-Diaz-Saa-inequ} and \eqref{f-decr-prop}, we obtain
\[
\begin{split}
   0&\leq   \mathcal{E}_L\l(w_2, \frac{w_2^q-w_1^q}{w_2^{q-1}}\r) - \mathcal{E}_L\l(w_1, \frac{w_2^q-w_1^q}{w_1^{q-1}}\r) \\
    & = \into (f(x,w_2) + \sigma w_2 \ln w_2) \frac{(w_2^q-w_1^q)}{w_2^{q-1}} ~dx - \into (f(x,w_1) + \sigma w_1 \ln w_1) \frac{(w_2^q-w_1^q)}{w_1^{q-1}} ~dx \\
    & = \int_{\Omega_0} \l(\frac{f(x,w_2) + \sigma w_2 \ln w_2}{w_2^{q-1}}-\frac{f(x,w_1) + \sigma w_1 \ln w_1}{w_1^{q-1}}\r) (w_2^q-w_1^q) ~dx <  0
\end{split}
\]
which yields a contradiction to the fact that $|\Omega_0| \neq 0$. Hence, we get the uniqueness of the solution.
 \qed
  \subsection{Uniform asymptotics estimates}
  In this subsection, we derive uniform estimates on the sequence of solution $u_s$ for the weighted fractional Dirichlet problem \eqref{Eq:Problem} under the following assumptions 
  \begin{align}
 \label{assumption-on-p1}
  \textbf{Sublinear growth:} \quad 1 < p(s) <2, \quad p'(0) \in (-\infty,0).
 \end{align}
  \begin{lemma}
      \label{Lemma3}
      Let $p$ satisfies \eqref{assumption-on-p1} such that $\lim_{s \to 0^+} p(s) =2.$ Then, for all $\ph \in C_c^\infty(\Omega)$, we have
      \[
      \begin{split}
          \lim_{s \to 0^+} \frac{E_{s}(\ph)}{s} &= p'(0) \frac{\|\ph\|^2_{L^2(\Omega)}}{4} + \frac{1}{2} \mathcal{E}_L(\ph,\ph) \\
          & - \frac{1}{2}\l(p'(0)\int_{\R^N} \abs{\ph}^2\ln |\ph| ~dx + \int_{\R^N}a'(0,x)\abs{\ph}^2 ~dx\r). 
      \end{split}
      \]
  \end{lemma}
  \begin{proof}
  Using the definition of energy functional $E_{s}$, $\|\ph\|_{s}^2 \to \|\ph\|^2_{L^2(\Omega)}$ as $s \to 0^+$ and \cite[Lemma 3.7]{Santamaria-Saldana-2022}, we obtain
  \[
  \begin{split}
  & \lim_{s \to 0^+} \frac{E_{s}(\ph)}{s} = \lim_{s \to 0^+} \frac{1}{2s}\|\ph\|^2_{s}- \frac{1}{s p(s)}\into a(s,x)\abs{\ph}^{p(s)} ~dx\\
    & \quad = \lim_{s \to 0^+} \frac{1}{s}\l(\frac{1}{2}-\frac{1}{p(s)}\r)\|\ph\|^2_{s} + \lim_{s \to 0^+} \frac{1}{p(s)s}\l(\|\ph\|^2_{s}- \into a(s,x)\abs{\ph}^{p(s)} ~dx\r)\\
    & \quad = \frac{p'(0) \|\ph\|^2_{L^2(\Omega)}}{4} + \frac{1}{2} \mathcal{E}_L(u,u) + \frac{1}{2} \lim_{s \to 0^+} \frac{1}{s}\l(\|\ph\|^2_{L^2(\Omega)}- \into a(s,x)\abs{\ph}^{p(s)} ~dx\r).
  \end{split}
  \]
 
Note that
  \[
  \begin{split}
  & \l(\into a(s,x)\abs{\ph}^{p(s)} ~dx-\|\ph\|^2_{L^2(\Omega)}\r) = \into \int_0^1 \frac{d}{d\theta} \l(a(\theta s,x) |\ph|^{2+ (p(s)-2)\theta} \r) ~d\theta ~dx\\
  & \quad = \int_0^1 \into s a'(\theta s, x) |\ph|^{2+(p(s)-2)\theta} ~dx \\
  & \qquad \qquad + (p(s)-2)\int_0^1 \into a(\theta s,x) |\ph|^{2+(p(s)-2)\theta} \ln(|\ph|) ~d\theta ~dx.
  \end{split}
  \]
 
 Now, by using dominated convergence theorem and passing $s \to 0^+$ in the above estimate, we have
 \[
 \begin{split}
     \lim_{s \to 0^+} & \frac{1}{2 s} \l(\into a(s,x)\abs{\ph}^{p(s)}-\|\ph\|^2_{L^2(\Omega)}\r) \\
     & = \frac{1}{2} \int_{\R^N} a'(0,x)\abs{\ph}^2 ~dx  + \frac{p'(0)}{2} \int_{\R^N} \abs{\ph}^2 \ln\abs{\ph} ~dx.
 \end{split}
 \]
Hence, combining the above estimates, we get the required claim.
  \end{proof}
Next, we prove the uniform $L^\infty$ upper estimate for the weak solution of \eqref{Eq:Problem}.
  \begin{proposition}
      \label{Prop1}
      Let $\Omega\subseteq \R^N$ be a bounded domain, $p$ satisfies \eqref{assumption-on-p1} such that $\lim_{s \to 0^+} p(s) =2$ and $a$ satisfies \eqref{a:regulrity-asym}, \eqref{a:regulrity-1-asym} and \eqref{cond:a-upperbound}. Let $u_{s}$ be the weak solution of \eqref{Eq:Problem}.
      Then, \[\|u_s\|_{L^\infty(\Omega)} \leq c_3{(R^2 e^{\frac{1}{2}-\rho_N})}^{\frac{-1}{p'(0)}} +o(1) \quad \text{as} \ s \to 0^+ \]
      where $R:= 2 \diam(\Omega)$.
  \end{proposition}
  \begin{proof}
      Denote $c_1 := \ln R^2+\frac{1}{2}-\rho_N$. Using \cite[Proposition 8.1]{Ros-oton-Serra-Valdinoci-2017}, \cite[Lemma 5]{Angeles-Saldana-2023} and \eqref{cond:a-upperbound}, we get
      \[
      \begin{split}
          \|u_s\|_{{L^\infty(\Omega)}} & \leq (1+s c_1+o(s))\|a(s,x)\abs{u_s}^{p(s)-2}u_s\|_{L^{N/s^2}(\Omega)}\\
          &= (1+s c_1+o(s)){\l(\into \abs{a(s,x)\abs{u_s}^{p(s)-2}u_s}^{N/s^2} ~dx\r)}^{s^2/N}\\
          &\leq (1+s c_1+o(s)) c_3^s \ \|u_s\|^{p(s)-1}_{{L^\infty(\Omega)}} \ \abs{\Omega}^{s^2/N}.
      \end{split}
      \]
This further implies      
      \[
      \begin{split}
         \|u_s\|_{{L^\infty(\Omega)}} &\leq  
         {\l((1+sc_1+o(s))^{\frac{1}{s}} c_3\abs{\Omega}^{s/N}\r)}^{\frac{s}{2-p(s)}}.\\
      \end{split}
      \]
Finally, by using \cite[Lemma 3.1]{Santamaria-Saldana-2022},
      \[
      \begin{split}
          \lim_{s\to 0} {\l((1+s c_1+o(s))^{\frac{1}{s}} c_3\abs{\Omega}^{s/N}\r)}^{\frac{s}{2-p(s)}} &= (c_3e^{c_1})^\frac{-1}{p'(0)} = (c_3{(R^2 e^{\frac{1}{2}-\rho_N})})^{\frac{-1}{p'(0)}}. 
      \end{split}
      \]
  \end{proof}
In what follows, $\ph_s$ represents the $L^2$-normalized Dirichlet eigenfunction of the fractional Laplacian and $\lambda_{1,s}$ represents the first eigenvalue of the fractional Laplacian. Also, by the variational characterization of the first eigenvalue of the fractional Laplacian, we have
  \begin{equation}
  \label{var-char-eigenvalue}
  \|u\|^2_{L^2(\Omega)} \leq \frac{1}{\lambda_{1,s}}\|u\|_s^2 \quad \text{for all} \ u\in H^s_0(\Omega), \ s\in (0,\frac{1}{4}).
  \end{equation}
To study the asymptotic behavior of solutions $u_s$, we derive upper and lower estimates for the $H^s(\Omega)$-norm of $u_s$.
  \begin{lemma}
      \label{Lemma6}
      Let $p$ satisfies $\lim_{s \to 0^+} p(s) =2$ and $u_s$ be the positive least energy solution of \eqref{Eq:Problem}. Then,
      \[\frac{2p(s)}{p(s)-2} E_{s}\l(\frac{t_s \ph_{s}}{2}\r) \leq \|u_s\|^2_s \leq  \l(c_3^s{c_s^{p(s)}}\r)^{\frac{2}{2-p(s)}}\l[\frac{1}{2}- \frac{1}{p(s)}\r],\]
      where
      \[t_s := {\l(\frac{2}{p(s)}\frac{\into a(s,x)\abs{\ph_{s}}^{p(s)}}{\lambda_{1,s}\|\ph_{s}\|_{L^2(\Omega)}^2}\r)}^{1/(2-p(s))} \quad \text{and} \quad c_s:= (\lambda_{1,s})^{-\frac{1}{2}}\abs{\Omega}^{\frac{2-p(s)}{2p(s)}}.\]
  \end{lemma}
  \begin{proof}
 If $t< t_s$, then
  \[
  \begin{split}
      E_{s}(t \ph_{s}) & = \frac{t^2}{2}\|\ph_{s}\|_{s}^2- \frac{t^{p(s)}}{p(s)}\into a(s,x)\abs{ \ph_{s}}^{p(s)} ~dx \\
      &= \frac{t^2}{2} \lambda_{1,s} \|\ph_{s}\|^2_{L^2(\Omega)}- \frac{t^{p(s)}}{p(s)}\into a(s,x)\abs{ \ph_{s}}^{p(s)} ~dx < 0.
  \end{split}
      \]
Since $u_s$ is positive least energy solution of \eqref{Eq:Problem}, we get
      \[
      \begin{split}
       E_{s}(u_s) &= \frac{1}{2}\|u_s\|_{s}^2- \frac{1}{p(s)}\into a(s,x)\abs{u_s}^{p(s)} ~dx = \l(\frac{1}{2}-\frac{1}{p(s)}\r)\|u_s\|^2_{s}\\
       & \leq  E_{s}\l(\frac{t_s \ph_{s}}{2}\r) = \frac{t_s^2}{8} \lambda_{1,s} \|\ph_{s}\|^2_{L^2(\Omega)}- \frac{t_s^{p(s)}}{2^{p(s)}p(s)}\into a(s,x)\abs{ \ph_{s}}^{p(s)} ~dx.
  \end{split}.
      \]
This gives 
\[ \|u_s\|^2_{s} \geq \frac{2p(s)}{p(s)-2} E_{s}\l(\frac{t_s \ph_{s}}{2}\r). \]
By \eqref{var-char-eigenvalue}, we have
      \[
      \begin{split}
        E_{s}(u_s)&=\frac{1}{2}\|u_s\|_{s}^2- \frac{1}{p(s)}\into a(s,x)\abs{u_s}^{p(s)} ~dx \geq \frac{1}{2}\|u_s\|_{s}^2- \frac{c_3^s}{p(s)} c_s^{p(s)}\|u_s\|_{s}^{p(s)}.
      \end{split}
      \]
      Defining for $t>0$
      \[ \theta(t) = \frac{1}{2}t^2 - \frac{c_3^{s}}{p(s)} {c_s^{p(s)}}t^{p(s)},\]
it is easy to verify that the function $\theta$ attains minimum at $t_0 = \l[c_3^{s}{c_s^{p(s)}}\r]^{\frac{1}{2-p(s)}}.$ Therefore,
      \[
      \begin{split}
         \|u_s\|^2_{s} \leq  \l(c_3^{s}{c_s^{p(s)}}\r)^{\frac{2}{2-p(s)}}\l[\frac{1}{2}- \frac{1}{p(s)} \r].
      \end{split}
      \]
\end{proof}
\begin{lemma}\label{limits:bounds}
    Let $p$ satisfies $\lim_{s \to 0^+} p(s) =2$ and $u_s$ be the positive least energy solution of \eqref{Eq:Problem}. Then, 
    \[\lim_{s \to 0^+} \|u_s\|_s^2 \geq \frac{\ln 2}{2} A,\]
    where
\[
\begin{split}
    A:= \exp\l(1 + \frac{2\lambda_{1,L}}{p'(0)} - \frac{2}{p'(0)} \l(\into a'(0,x)\abs{\ph_L}^2 ~dx + p'(0) \into \abs{\ph_L}^2 \ln\abs{\ph_L} ~dx\r) \r).
\end{split}
\]
\end{lemma}
\begin{proof} 
We first estimate 
\[\lim_{s \to 0^+} t_s^2 = \lim_{s \to 0^+} {\l(\frac{2}{p(s)}\frac{\into a(s,x)|\ph_{s}|^{p(s)} ~dx}{\lambda_{1,s}\|\ph_{s}\|_{L^2(\Omega)}^2}\r)}^{\frac{2}{2-p(s)}}.\]
Note that
\[
\begin{split}
    & \frac{\into a(s,x)|\ph_{s}|^{p(s)} ~dx-\|\ph_{s}\|_{L^2(\Omega)}^2}{s} \\
    & \qquad \qquad \to  \into a'(0,x)\abs{\ph_L}^2 ~dx + p'(0) \into \abs{\ph_L}^2 \ln\abs{\ph_L} ~dx \quad \text{as} \  s \to 0^+.
\end{split}
\]
Applying \cite[Lemma 3.7]{Santamaria-Saldana-2022} and using the fact that $\|\ph_L\|_{L^2(\Omega)}^2 =1$, we have
\[{\l(\frac{2}{p(s)}\r)}^{\frac{2}{2-p(s)}} = {\bigg[1 - s \frac{p'(0)}{2}+o(s)\bigg]}^{\frac{2}{2-p(s)}} \to e \quad \text{as} \ s \to 0^+,\]
\[{\l(\frac{1}{\lambda_{1,s}}\r)}^{\frac{2}{2-p(s)}}= {\bigg[1 - s \lambda_{1,L}+o(s)\bigg]}^{\frac{2}{2-p(s)}} \to \exp\l(\frac{2\lambda_{1,L}}{p'(0)}\r)\quad \text{as} \ s \to 0^+\]
and
\[
\begin{split}
    & \l(\frac{\into a(s,x)|\ph_{s}|^{p(s)}~dx}{\|\ph_{s}\|_{L^2(\Omega)}^2}\r)^{\frac{2}{2-p(s)}}\\
    & \qquad = \bigg[1- \frac{s}{\|\ph_L\|_{L^2(\Omega)}^2 + o(1)}\l(\into a'(0,x)\abs{\ph_L}^2 ~dx + p'(0) \into \abs{\ph_L}^2 \ln\abs{\ph_L} ~dx\r)\\ & \qquad \qquad  \qquad \qquad +o(s)\bigg]^{\frac{2}{2-p(s)}}\\
    & \qquad \xrightarrow{s \to 0^+} \exp{\l(- \frac{2}{p'(0)}\l(\into a'(0,x)\abs{\ph_L}^2 ~dx + p'(0) \into \abs{\ph_L}^2 \ln\abs{\ph_L} ~dx\r) \r)}.
\end{split}
\]
Collecting the above estimates, we obtain $\lim_{s \to 0^+} t_s^2 = A$ where 
\[
\begin{split}
    A:= \exp\l(1 + \frac{2\lambda_{1,L}}{p'(0)} - \frac{2}{p'(0)} \l(\into a'(0,x)\abs{\ph_L}^2 ~dx + p'(0) \into \abs{\ph_L}^2 \ln\abs{\ph_L} ~dx\r) \r).
\end{split}
\]
From Lemma \ref{Lemma6}, we have
\[\frac{2p(s)}{p(s)-2} E_{s}\l(\frac{t_s \ph_{s}}{2}\r) \leq \|u_s\|^2_{s} \leq {[c_3^{s} c_{s}^{p(s)}]}^{\frac{2}{2-p(s)}}\bigg[\frac{1}{2}-\frac{1}{p(s)}\bigg].\]
Passing to the limit as $s \to 0^+$ in lower bound of the preceding inequality, we get
\[
\begin{split}
\lim_{s \to 0^+} \|u_s\|_s^2 & \geq \lim_{s \to 0^+} \frac{2p(s)}{p(s)-2} E_{s}\l(\frac{t_s \ph_{s}}{2}\r)\\
& = \lim_{s \to 0^+} \frac{\bigg[\lambda_{1,s}\|\ph_s\|_{L^2(\Omega)}^2 \  t_s^2\l(\frac{1}{8}-\frac{1}{2^{p(s)+1}}\r)\bigg]}{\l(\frac{1}{2}-\frac{1}{p(s)}\r)} = \frac{\ln 2}{2} A.
\end{split}
\]
\end{proof}
  \subsection{Asymptotics for non-local sublinear problem}
In this subsection, we study the asymptotics for the non-local problem \eqref{Eq:Problem} with sublinear growth, {\it i.e.,} $p(\cdot)$ satisfies \eqref{assumption-on-p1}. 
  \begin{theorem}
      \label{Theorem7}
       Let $p$ satisfies \eqref{assumption-on-p1} such that $\lim_{s \to 0^+} p(s) =2$ and $a$ satisfies \eqref{a:regulrity-asym}, \eqref{a:regulrity-1-asym} and \eqref{cond:a-upperbound}. Let $u_s$ be a positive least energy solution of \eqref{Eq:Problem}. Then, there is a constant $C$ depending on $\Omega$, $p'(0)$ and independent of $s$ such that
       \[\|u_s\|^2= \mathcal{E}(u_s,u_s)\leq C+o(1) \ \text{as} \ s \to 0^+.\]
  \end{theorem}
  \begin{proof}
      Note by \cite[Lemma 3.5]{Santamaria-Saldana-2022}, $u_s \in \h(\Omega)$, for all $s \in (0, \frac{1}{4}).$ For a fixed $s \in (0, \frac{1}{4})$, let $\{\ph_n\}$ be a sequence of $C^\infty_c(\Omega)$ functions such that $\ph_n \to u_s$ in $H^{s}_0(\Omega)$ as $n \to \infty$. Denote
      \[
      I_n := \frac{\|\ph_n\|^2_{s}-\|\ph_n\|^2_{L^2(\Omega)}}{s} = \int_0^1 \int_{\R^N} \abs{\xi}^{2{s}\tau} \ln \abs{\xi}^2 \abs{\widehat{\ph_n}(\xi)}^2 ~d\xi ~d\tau.
          \]
Since $u_s$ is the weak solution of \eqref{Eq:Problem}, we have          
          \[
      \begin{split}
          I_n &=\frac{1}{s}\l[ 2 E_{s}(\ph_n)+\frac{2}{p(s)}\into a(s,x)\abs{\ph_n}^{p(s)} ~dx \r]- \frac{\|\ph_n\|^2_{L^2(\Omega)}}{s} \\
          &= \frac{1}{s}\l[2E_{s}(\ph_n)+\frac{2-p(s)}{p(s)}\into a(s,x)\abs{\ph_n}^{p(s)} ~dx\r]\\
         & \qquad  + \frac{1}{s}\l[\into a(s,x)\abs{\ph_n}^{p(s)} ~dx-\into \abs{\ph_n}^2 ~dx\r].
      \end{split}
      \]
Passing to the limit as $n \to \infty$, we obtain
      \[
      \begin{split}
      \lim_{n \to \infty} & 2 E_{s}(\ph_n) + \frac{2-p(s)}{p(s)}\into a(s,x)\abs{\ph_n}^{p(s)} ~dx \\
      &= 2E_{s}(u_s) +\frac{2-p(s)}{p(s)}\into a(s,x)\abs{u_s}^{p(s)} ~dx = 0.
      \end{split}
      \]
      Therefore, we have
      \[
      \begin{split}
      I_n &= \frac{1}{s}\l[\into a(s,x)\abs{\ph_n}^{p(s)} ~dx-\into \abs{\ph_n}^2 ~dx\r]+o(1)\\
      &=\into \int_0^1 a'(s\tau, x) \abs{\ph_n}^{2+(p(s)-2)\tau} ~dx ~d\tau\\
      & \quad + \l(\frac{p(s)-2}{s}\r)\int_0^1 \into a(s\tau,x)\abs{\ph_n}^{2+(p(s)-2)\tau} \ln \abs{\ph_n} ~dx ~d\tau + o(1)\\
      &:= I_1+I_2 + o(1).
      \end{split}
      \]
Note that   
      \[I_1 \to \into a'(0,x)\abs{\ph_n}^2 ~dx  \quad \text{as} \ s \to 0^+.\]
Using $1< p(s) <2$, \eqref{a:regulrity-asym} and \eqref{a:regulrity-1-asym} for $s$ small enough, we get
      \[
  \begin{split} 
      I_2 & = \l(\frac{p(s)-2}{s}\r) \bigg[\int_0^1 \int_{\{\abs{\ph_n}<1\}} a(s\tau,x)\abs{\ph_n}^{2+(p(s)-2)\tau}\ln \abs{\ph_n} ~dx~d\tau\\
       &\qquad +\int_0^1 \int_{\{\abs{\ph_n}\geq1\}} a(s\tau,x)\abs{\ph_n}^{2+(p(s)-2)\tau} \ln |\ph_n| ~dx ~d\tau\bigg]\\
        &\leq \l(\frac{p(s)-2}{s}\r)\bigg[\int_0^1 \int_{\{\abs{\ph_n}<1\}} a(s\tau,x)\abs{\ph_n}^{2+(p(s)-2)\tau} \ln \abs{\ph_n} ~dx ~d\tau\bigg]\\
        &\leq \l(\frac{2-p(s)}{s}\r)\abs{\Omega}c_3^{s}\sup_{t\in (0,1)} \abs{t}\abs{\ln \abs{t}}\\
        &< \l(\frac{2-p(s)}{s}\r)c_3^{s}\abs{\Omega}.
  \end{split}  
      \]
Therefore,
      \[
      I_2 \to - p'(0) \abs{\Omega} \quad \text{as} \ s \to 0^+.
      \]
On the other hand, using the definition of $\mathcal{E}_L(\cdot,\cdot)$, we get
      \[
      \begin{split}
      I_n &\geq  \int_0^1 \int_{\{\abs{\xi}<1\}} \abs{\xi}^{2{s_k}\tau} \ln \abs{\xi}^2 \abs{\widehat{\ph_n}(\xi)}^2 ~d\xi ~d\tau + \int_{\{\abs{\xi}\geq 1\}} \ln \abs{\xi}^2 \abs{\widehat{\ph_n}(\xi)}^2 ~d\xi \\
      &= \int_0^1 \int_{\{\abs{\xi}<1\}} \abs{\xi}^{2{s_k}\tau} \ln \abs{\xi}^2 \abs{\widehat{\ph_n}(\xi)}^2 ~d\xi -\int_{\{\abs{\xi}< 1\}} \ln \abs{\xi}^2 \abs{\widehat{\ph_n}(\xi)}^2 ~d\xi ~d\tau \\
      & \qquad +\int_{\R^N} \ln \abs{\xi}^2 \abs{\widehat{\ph_n}(\xi)}^2 ~d\xi \\
      &\geq \int_{\R^N} \ln \abs{\xi}^2 \abs{\widehat{\ph_n}(\xi)}^2 ~d\xi \geq \mathcal{E}_L(\ph_n,\ph_n) \geq \|\ph_n\|^2- C_1 \|\ph_n\|^2_{L^2(\Omega)},
      \end{split}
      \]
      where $C_1>0$ is a constant that depends on $\Omega.$ 
      Finally, by using Proposition \ref{Prop1} and combining the above estimates, there exists $C_2$ (independent of $n$) such that
      \[
      \begin{split}
      \|\ph_n\|^2 \leq I_n + C_1 \|\ph_n\|^2_{L^2(\Omega)}  \leq C_2+o(1).
      \end{split}
      \]
      As $n \to \infty$, we get the required bound on $\|u_s\|$.
  \end{proof}
\noindent   \textit{Proof of Theorem \ref{thm-asym-sub}:}
      By Theorem \ref{Theorem7}, $u_s$ is uniformly bounded in $\h(\Omega)$. This implies that there exists a $u_0 \in \h(\Omega)$ such that $u_s \rightharpoonup u_0$ in $\h(\Omega)$, which by the compact embedding $\h(\Omega)\hookrightarrow L^2(\Omega)$ implies $u_s \to u_0$ in $L^2(\Omega)$ and $u_s\to u_0$ a.e. in $\Omega$.\\
      \textbf{Claim 1:} $u_0$ is a weak solution of \eqref{Eq:LimitingProblem}.\\
      Let $\ph \in C^\infty_c(\Omega)$. By \cite[Theorem 1.1]{Chen-Weth-2019}, we have
      \[
  \begin{split}
      \into & u_s (\ph + s L_\Delta \ph +o(s)) ~dx \\
      &= \into u_s (-\Delta)^{s} \ph ~dx = \into a(s,x)\abs{u_s}^{p(s)-2} u_s \ph ~dx\\
      &=\into u_{s} \bigg(1+ s \int_0^1 \bigg[ \big[a'(s \tau,x) +a(s \tau, x)\ln \abs{u_{s}}p'(s\tau)\big]\abs{u_{s}}^{p(s\tau )-2}\bigg] ~d\tau+o(s) \bigg)\ph ~dx
  \end{split}
        \]
holds in $L^\infty(\Omega).$ Therefore,  
  \[
  \begin{split}
  \mathcal{E}_L(u_s,\ph)+o(1) &= \into u_s L_\Delta\ph ~dx +o(1)\\
  &= \into  \int_0^1 \big[a'(s \tau,x) +a(s \tau, x)\ln \abs{u_{s}}p'(s\tau)\big]\abs{u_{s}}^{p(s\tau )-2} u_{s} \ph ~d\tau  ~dx.
  \end{split}
  \]
  Letting $s \to 0^+$, using \eqref{a:regulrity-1-asym} and \cite[Lemma 2]{Angeles-Saldana-2023}, we get
  \[
  \begin{split}
      \mathcal{E}_L(u_0, \ph)&= p'(0) \into \ln \abs{u_0} u_0\ph ~dx+ \into a'(0,x) u_0\ph ~dx.
  \end{split}
  \]
  Hence, by using density of $C_c^\infty(\Omega)$ in $\h(\Omega)$, we infer claim 1.\\
  \textbf{Claim 2:} $u_0 \not \equiv 0.$\\
  By Lemma \ref{limits:bounds}, there exists $C>0$ depending on $\Omega$ and $p'(0)$ such that
  \[C \leq \|u_s\|^2_{s} = \into a(s,x)\abs{u_s}^{p(s)} ~dx \leq c_3^{s}\abs{\Omega}^{\frac{2-p(s)}{2}}{\l(\into \abs{u_s}^2 ~dx\r)}^{\frac{p(s)}{2}}.\]
  Letting $s \to 0^+$, we obtain $0 < C \leq \|u_0\|_{L^2(\Omega)}^2$ which implies that $u_0 \not \equiv 0.$ Moreover, we have $\mathbb{E}(u_0)=  \frac{p'(0)}{4}\into u_0^2 ~dx.$\\
  \textbf{Claim 3:} $u_0$ is a least energy solution of \eqref{Eq:LimitingProblem}. Using \eqref{a:regulrity-asym}, we get
  \[
  \begin{split}
      \frac{p'(0)}{4} \lim_{s \to 0^+} \|u_s\|_{s}&= \frac{p'(0)}{4}\lim_{s \to 0^+} \into a(s,x)\abs{u_s}^{p(s)} ~dx\\
      &= \frac{p'(0)}{4} \|{u_0}\|^2_{L^2(\Omega)} = \mathbb{E}(u_0).
  \end{split}
  \]
  By Theorem \ref{them:existence-limiting}, there exists $v_0 \in \h(\Omega)$ such that $\mathbb{E}(v_0)= \inf_{\h(\Omega)} \mathbb{E}$. By density, let $(v_k)_{k\in \N}\subset C^\infty_c(\Omega)$ such that $v_k \to v_0$ in $\h(\Omega)$ as $k\to \infty$. Since $u_s$ is a least energy solution, for every $k \in \mathbb{N}$, we have
  \[
  \frac{p'(0)}{4} \lim_{s \to 0^+} \|u_s\|^2_{s}= \lim_{s \to 0^+} \l(\frac{p(s)-2}{2p(s)}\r) \frac{\|u_s\|^2_{s}}{s} = \lim_{s \to 0^+} \frac{E_{s}(u_s)}{s}\leq \frac{E_{s}(v_k)}{s}.\]
  By Lemma \ref{Lemma3}, as $k \to \infty$,
  \[\inf_{\h(\Omega)}  \mathbb{E} \leq \mathbb{E}(u_0) = \frac{p'(0)}{4} \|u_0\|_{L^2(\Omega)}^2 = \frac{p'(0)}{4} \lim_{s \to 0^+} \|u_s\|^2_{s}\leq \inf_{\h(\Omega)}  \mathbb{E}\]
  which implies that $u_0$ is a least energy solution of \eqref{Eq:LimitingProblem}. Since $u_s$ is a sequence of positive solutions of \eqref{Eq:Problem}, we can conclude that $u_0$ is non-negative.\\
  \textbf{Claim 4:} $u_0 \in L^\infty(\Omega)$ and $\|u_0\|_{L^\infty(\Omega)} \leq M{(R^2 e^{\frac{1}{2}-\rho_N})}^{\frac{1}{p'(0)}}:= C_0$, where $R= 2$diam$(\Omega)$.\\
  By Proposition \ref{Prop1}, $\|u_s\|_{L^\infty(\Omega)} \leq C_0+o(1)$ as $s \to 0^+$. We argue by contradiction. Assume that there exists $\eps>0$ and $\Omega_1 \subset \Omega$ of positive measure such that $\abs{u_0}>(1+\eps)C_0$ in $\Omega_1$.
  This implies, for a.e. $x\in \Omega_1$,
  \[\abs{u_s(x)-u_0(x)}\geq \abs{u_0(x)}-\abs{u_s(x)}>\eps C_0,\]
  which in turn implies
  \[\into \abs{u_s-u_0}^2 ~dx \geq \int_{\Omega_1} \abs{u_s-u_0}^2 ~dx > \eps C_0\abs{\Omega_1}>0,\]
  contradicting the fact that $u_s\to u_0$ in $L^2(\Omega)$. Therefore, 
  \begin{equation}
    \label{estimate-on-L-infinity-norm}
  \|u_0\|_{L^\infty(\Omega)}\leq M{(R^2 e^{\frac{1}{2}-\rho_N})}^{\frac{1}{p'(0)}}
  \end{equation}
and thus, by dominated convergence theorem, $u_s \to u_0$ in $L^q(\Omega)$ for any $1\leq q <\infty.$ Finally, by applying Theorem \ref{thm:regularity-uniqueness}, we obtain $u_0$ satisfies \eqref{u_reg-sublinear}.\qed 
\section{Final comments and open problems}
In this section, we outline some contrasting features of our Brezis-Nirenberg problem involving the logarithmic Laplacian in comparison with the Brezis-Nirenberg problem for both the Laplacian and the fractional Laplacian. We also propose some interesting open problem for future research.
\begin{itemize}
    \item The existence result of non-trivial weak solution for the Brezis-Nirenberg problem \eqref{prob:logarithmic-laplace} with $f \equiv \lambda$ does depend upon the range of $\lambda$, while in the case of Brezis-Nirenberg problem for the Laplacian and the fractional Laplacian for $\lambda <0$, the non-existence of positive solution is known via Pohozaev type identity. 
    \item For our problem \eqref{prob:logarithmic-laplace}, the ground state of weak solution is proved by minimizing the energy $\mathbb{E}$ over the Nehari manifold $N_{0, \sigma}$ and exploiting the compact embedding of $\h(\Omega)$ in $L^2(\Omega)$. While the same arguments cannot be applied to the Brezis-Nirenberg problem for the Laplacian and the fractional Laplacian due to the lack of compactness in the embeddings of the corresponding energy spaces. 
    \item It is worth to emphasize that unlike the results available for the Brezis-Nirenberg Problem for the Laplacian \cite{Brezis-Nirenberg-1983}, the fractional Laplacian \cite{Servadei-Valdinoci-2013-1, Servadei-Valdinoci-2015, Servadei-Valdinoci-2013} and the fractional $p$-Laplacian \cite{{Mosconi-Perera-Squassina-Yang-2016}}, in our case we do not have any restrictions on the dimension $N$ when we consider \eqref{prob:logarithmic-laplace} with $\sigma \in (0,\frac{4}{N}).$ 
    \item The condition $\sigma < \frac{4}{N}$ in Theorem \ref{Main-res-limitingprob}  plays a significant role (see Proposition \ref{weak-conve-est}). The existence or non-existence of weak solution when $\sigma \geq \frac{4}{N}$ is an open problem.
\end{itemize}
\section*{Acknowledgments}

The first author is funded by the Scientific High-level visiting fellowship (SSHN-2024) by the French Institute in India (IFI) and the Embassy of France in India. The first author thanks the Universit\'e de Pau et des Pays de l'Adour, Pau for the kind hospitality during a research stay in October 2024. 

The second author is partially funded by IFCAM (Indo-French Centre for Applied Mathematics) IRL CNRS 3494.

The third author is funded by the UGC Junior Research Fellowship with reference no. 221610015405. 



\end{document}